\documentclass[a4paper,10pt]{article}
\usepackage{amssymb,amsmath,amsthm,mathrsfs}
\usepackage{fullpage}
\usepackage{hyperref}
\usepackage{eucal}
\usepackage{stackrel}
\newcommand{\ie}{\emph{i.e.}}
\newcommand{\eg}{\emph{e.g.}}
\newcommand{\cf}{\emph{cf.}}
\newcommand{\Real}{\mathbb{R}}
\newcommand{\Com}{\mathbb{C}}
\newcommand{\Nat}{\mathbb{N}}

\newcommand{\sgn}{\mathop{\mathrm{sgn}}\nolimits}
\newcommand{\supp}{\mathop{\mathrm{supp}}\nolimits}
\newcommand{\Dom}{\mathsf{D}}
\newcommand{\Ran}{\mathsf{R}}

\newcommand{\curl}{\mathop{\mathrm{curl}}\nolimits}
\newcommand{\eps}{\varepsilon}
\newcommand{\sii}{L^2}
\newcommand{\vertiii}[1]{{\left\vert\kern-0.25ex\left\vert\kern-0.25ex\left\vert #1 
    \right\vert\kern-0.25ex\right\vert\kern-0.25ex\right\vert}}
\newcommand{\nspace}{\mathcal{D}}
\newtheorem{Theorem}{Theorem}
\newtheorem{Lemma}{Lemma}

\newtheorem{Corollary}{Corollary}

\theoremstyle{definition}
\newtheorem{Remark}{Remark}

\usepackage[normalem]{ulem}
\usepackage{color}
\definecolor{DarkBlue}{rgb}{0,0.1,0.7}

\newcommand\soutD{\bgroup\markoverwith
{\textcolor{DarkBlue}{\rule[.01ex]{2pt}{1pt}}}\ULon}
\newcommand{\Hm}[1]{\leavevmode{\marginpar{\tiny%
$\hbox to 0mm{\hspace*{-0.5mm}$\leftarrow$\hss}%
\vcenter{\vrule depth 0.1mm height 0.1mm width \the\marginparwidth}%
\hbox to
0mm{\hss$\rightarrow$\hspace*{-0.5mm}}$\\\relax\raggedright #1}}}

\begin{document}
%
\title{\textbf{
Spectral stability of Schr\"odinger operators
with subordinated complex potentials 
}}
\author{Luca Fanelli,$^{a}$ \ David Krej\v{c}i\v{r}\'ik$^{b}$ \ and \ Luis Vega$^{c}$}
\date{\small 
\emph{
\begin{quote}
\begin{itemize}
\item[$a)$] 
Dipartimento di Matematica, SAPIENZA Universit\`a di Roma,
P.~le Aldo Moro 5, 00185 Roma;
fanelli@mat.uniroma1.it.%
\item[$b)$] 
Department of Mathematics, Faculty of Nuclear Sciences and 
Physical Engineering, Czech Technical University in Prague, 
Trojanova 13, 12000 Prague 2, Czechia;
david.krejcirik@fjfi.cvut.cz.%
\\
\item[$c)$]
Departamento de Matem\'aticas, Universidad del Pais Vasco, 
Aptdo.~644, 48080 Bilbao, \&
Basque Center for Applied Mathematics (BCAM), 
Alameda Mazarredo 14, 48009 Bilbao, Spain; 
luis.vega@ehu.es \& lvega@bcamath.org.
\end{itemize}
\end{quote}
}
\smallskip
6 December 2016}

\maketitle

\begin{abstract}
\noindent
We prove that the spectrum of Schr\"odinger operators in three dimensions
is purely continuous and coincides with the non-negative semiaxis
for all potentials satisfying a form-subordinate smallness condition.
By developing the method of multipliers, 
we also establish the absence of point spectrum for 
Schr\"odinger operators in all dimensions
under various alternative hypotheses,
still allowing complex-valued potentials with critical singularities.
%
%
\end{abstract}
%

%
%
\section{Introduction}\label{Sec.intro}
Let~$H_0$ be the \emph{free Hamiltonian},
\ie~the self-adjoint operator in $\sii(\Real^d)$
associated with the quadratic form 
\begin{equation*}
  h_0[\psi] := \int_{\Real^d} |\nabla\psi|^2
  \,, \qquad
  \Dom(h_0) := H^1(\Real^d)
  \,.
\end{equation*}
Let $V:\Real^d\to\Com$ be a measurable function 
which  is \emph{form-subordinated} to~$H_0$
with the subordination bound less than one, \ie,
\begin{equation}\label{Ass}
  \exists a < 1
  \,, \quad
  \forall \psi \in H^1(\Real^d) 
  \,, \qquad
  \int_{\Real^d} |V| |\psi|^2 
  \leq a \int_{\Real^d} |\nabla\psi|^2
\end{equation}
In view of the criticality of~$H_0$ in low dimensions,
\eqref{Ass} is admissible for $d \geq 3$ only,
to which we restrict in the sequel. 

Assumption~\eqref{Ass} in particular means that  
the quadratic form 
\begin{equation}\label{eq:newform}
  v[\psi] := \int_{\Real^d} V \, |\psi|^2
  \,, \qquad
  \Dom(v) := \left\{
  \psi \in \sii(\Real^d):\int_{\Real^d} |V| |\psi|^2 < \infty
  \right\}
  \,.
\end{equation}
is relatively bounded with respect to~$h_0$
with the relative bound less than one.
Consequently, the sum $h_V:=h_0+v$ is a closed form 
with $\Dom(h_V) = H^1(\Real^d)$
which gives rise to an m-sectorial operator~$H_V$ in $\sii(\Real^d)$
via the representation theorem (\cf~\cite[Thm.~VI.2.1]{Kato}).  
It is customary to write 
\begin{equation}\label{form.sum}
  H_V = H_0 \dot{+} V
  \,,
\end{equation}
but we stress that this generalised sum in the sense of forms
differs from the ordinary operator sum. 

The purpose of this paper is to show that condition~\eqref{Ass}
is sufficient to guarantee that the spectra of~$H_0$ and~$H_V$ coincide, 
at least under some extra hypotheses.

Recall that the spectrum, $\sigma(H)$,
of a closed operator~$H$ in a complex Hilbert space~$\mathcal{H}$
is determined by the set of points $\lambda \in \Com$ for which 
$H-\lambda:\Dom(H)\to\mathcal{H}$ is not bijective.
Three disjoint subsets of $\sigma(H)$ that exhaust the spectrum
are distinguished:
the \emph{point spectrum} 
$
  \sigma_\mathrm{p}(H) := \{\lambda \in \Com : 
  H-\lambda \mbox{ is not injective} \}
$,
the \emph{continuous spectrum}
$
  \sigma_\mathrm{c}(H) 
  := \{\lambda \in \sigma(H) \setminus \sigma_\mathrm{p}(H) : 
  \overline{\Ran(H-\lambda)} = \mathcal{H} \}
$
and the \emph{residual spectrum}
$
  \sigma_\mathrm{r}(H) 
  := \{\lambda \in \sigma(H) \setminus \sigma_\mathrm{p}(H) : 
  \overline{\Ran(H-\lambda)} \not= \mathcal{H} \}
$.

The spectrum of~$H_0$ is well known to be purely continuous, in fact 
 $
  \sigma(H_0)
  = \sigma_\mathrm{c}(H_0) 
  = [0,+\infty)
$. 
In this paper we show that this spectral property is preserved 
by condition~\eqref{Ass} provided that $d=3$. 

\begin{Theorem}\label{Thm.main}
Let $d=3$ and assume~\eqref{Ass}. 
Then 
$
  \sigma(H_V)
  = \sigma_\mathrm{c}(H_V) 
  = [0,+\infty)
$.
\end{Theorem}

The theorem is proved in four steps:
\begin{enumerate}
\setlength\itemsep{0em}
\item[(i)]
Absence of the residual spectrum; \hfill
Section~\ref{Sec.intro}.
\item[(ii)]
Absence of the point spectrum; \hfill 
Section~\ref{Sec.evs}.
\item[(iii)]
Absence of the continuous spectrum in $\Com \setminus [0,+\infty)$; \hfill 
Section~\ref{Sec.cont1}.
\item[(iv)]
Inclusion of $[0,+\infty)$ in the spectrum; \hfill
Section~\ref{Sec.cont2}.
\end{enumerate}
Property~(i) follows at once (in any dimension):
Since the adjoint operator satisfies 
$H_V^* = H_{\overline{V}} = \mathcal{T} H_V \mathcal{T}$,
where~$\mathcal{T}$ is the complex-conjugation operator
defined by $\mathcal{T}\psi:=\overline{\psi}$,
$H_V$ is \emph{$\mathcal{T}$-self-adjoint} (\cf~\cite[Sec.~III.5]{Edmunds-Evans})
and as such it has no residual spectrum (\cf~\cite{BK}).
The absence of eigenvalues~(ii) is established in Section~\ref{Sec.evs}
by means of an argument reminiscent of the Birman-Schwinger principle,
but we emphasise that positive eigenvalues are excluded as well. 
Property~(iii) is proved by a modified version of the previous argument
in Section~\ref{Sec.cont1}.
Finally, in Section~\ref{Sec.cont2},
we establish~(iv) 
with help of an abstract quadratic-form criterion 
for the inclusion of points in the spectrum.

The present paper is primarily motivated by a recent interest
in spectral theory of Schr\"odinger operators with complex potential,
see \cite{Abramov-Aslanyan-Davies_2001,
Frank-Laptev-Lieb-Seiringer_2006,Bruneau-Ouhabaz_2008,
Laptev-Safronov_2009,Demuth-Hansmann-Katriel_2009,Safronov_2010,
Frank_2011,Demuth-Hansmann-Katriel_2013,Enblom1,Frank-Simon}.
However, the role of hypothesis~\eqref{Ass}
to have the conclusion of Theorem~\ref{Thm.main}
seems to be new in the self-adjoint case, too.

As a matter of fact, Simon established 
the absence of eigenvalues in the self-adjoint case for $d=3$
already in~\cite[Thm.~III.12]{SiQF} (see also~\cite[Thm.~XIII.21]{RS4})
by assuming 
\begin{equation}\label{Rollnik}
  \|V\|_R^2 := 
  \iint_{\Real^3\times\Real^3} 
  \frac{|V(x)| |V(y)|}{|x-y|^2} \, dx \, dy < (4\pi)^2
  \,.
\end{equation}
The extension of his method to complex potentials is straightforward.
However, notice that our assumption~\eqref{Ass} is weaker. 
Indeed, \eqref{Ass}~is equivalent to~\eqref{b1},   
while
\begin{equation}\label{Simon.weak}
  \big\||V|^{1/2}H_0^{-1/2}\big\|^2 
  = \big\||V|^{1/2}H_0^{-1}|V|^{1/2}\big\|
  \leq \big\||V|^{1/2}H_0^{-1}|V|^{1/2}\big\|_\mathrm{HS}
  = \frac{\|V\|_R}{4\pi}
  \,,
\end{equation}
where~$\|\cdot\|$ and~$\|\cdot\|_\mathrm{HS}$
denote the operator and Hilbert-Schmidt norms in $\sii(\Real^3)$, respectively.
The last equality in~\eqref{Simon.weak} follows with help of the explicit
formula for the Green function~\eqref{Green} in~$\Real^3$.

To be more specific, notice that,
by virtue of the \emph{classical Hardy inequality}
\begin{equation}\label{Hardy}
  \forall \psi \in H^1(\Real^d)
  \,, \qquad
  \int_{\Real^d} |\nabla \psi|^2
  \geq \left(\frac{d-2}{2}\right)^2 \int_{\Real^d} 
  \frac{|\psi(x)|^2}{|x|^2} \, dx
  \,,
\end{equation}
our hypothesis~\eqref{Ass} is in particular 
satisfied for potentials~$V$ verifying 
\begin{equation}\label{Hardy.sufficient}
  |V(x)| \leq a \left(\frac{d-2}{2}\right)^2 \frac{1}{|x|^2}
\end{equation}
for almost every $x \in \Real^d$.
However, the Hardy potential on the right hand side of this inequality
does not even belong to the Rollnik class
characterised for $d=3$
by the norm~$\|\cdot\|_R$ in~\eqref{Rollnik}.
Furthermore, the location of the continuous spectrum 
without the hypothesis that~$V$ belongs to the Rollnik class
(which ensures the finiteness of the Hilbert-Schmidt norm above)
is less evident in our more general setting.

Our Theorem~\ref{Thm.main} 
is also an improvement upon the non-self-adjoint situation
considered by Frank in \cite[Thm.~2]{Frank_2011}.
First, he establishes the absence
of eigenvalues outside $[0,+\infty)$ only.
Second, his assumption to get the conclusion 
of Theorem~\ref{Thm.main} for $d=3$ is
\begin{equation}\label{Frank.condition}
  \int_{\Real^3} |V(x)|^{3/2} \, dx  < \frac{3^{3/2}}{4 \pi^2}
  \,,
\end{equation}
which is again stronger than ours~\eqref{Ass}.
Indeed, by the H\"older and Sobolev inequalities, 
\begin{equation}
  \int_{\Real^3} |V| |\psi|^2 
  \leq \left(\int_{\Real^3} |V|^{3/2}\right)^{2/3} 
  \left(\int_{\Real^3} |\psi|^{6}\right)^{1/3}
  \leq \left(\int_{\Real^3} |V|^{3/2}\right)^{2/3} 
  \frac{2^{4/3}}{3 \pi^{4/3}} \int_{\Real^3} |\nabla\psi|^{2}
  \,,
\end{equation}
for all $\psi \in H^1(\Real^3)$.
As an example, 
the Hardy potential on the right hand side of~\eqref{Hardy.sufficient}
makes the left hand side of~\eqref{Frank.condition} infinite,
while it is an admissible potential for our Theorem~\ref{Thm.main}.
Finally, let us mention that Frank and Simon have noticed recently
in~\cite{Frank-Simon} that even positive eigenvalues can be excluded.

Our hypothesis~\eqref{Ass} 
is of course intrinsically a smallness condition about~$V$. 
But it is interesting to notice
that it involves potentials with quite rough local singularities, 
\eg~\eqref{Hardy.sufficient}.
It seems that such potentials are not typically covered by previous works 
on the exclusion of embedded eigenvalues, even in the self-adjoint case; 
see \cite{Ionescu-Jerison_2003,Koch-Tataru_2006}
to quote just the most recent results based on
Carleman's estimates.

The extension of Theorem~\ref{Thm.main} to higher dimensions is not obvious,
since our method relies on the pointwise inequality
for Green's functions~\eqref{crucial},
which does not hold for $d > 3$.
As an alternative approach, in Section~\ref{Sec.mult}, 
we develop the technique of multipliers for Schr\"odinger
operators with complex-valued potentials and prove 
the absence of eigenvalues in any dimension
under a stronger hypothesis.
\begin{Theorem}\label{Thm.any}
Let $d \geq 3$ and assume
\begin{equation}\label{Ass.bis}
  \exists b < \frac{d-2}{5d-8}
  \,, \quad
  \forall \psi \in H^1(\Real^d) 
  \,, \qquad
  \int_{\Real^d} r^2 \, |V|^2 \, |\psi|^2 
  \leq b^2 \int_{\Real^d} |\nabla\psi|^2
  \,,
\end{equation}
where $r(x):=|x|$.
Then 
$
  \sigma_\mathrm{p}(H_V) = \varnothing
$.
\end{Theorem}

Notice that~\eqref{Ass} follows as a consequence of~\eqref{Ass.bis}
by means of the Schwarz inequality 
and the classical Hardy inequality~\eqref{Hardy}. 
Indeed, \eqref{Ass.bis} and~\eqref{Hardy} yield
\begin{equation}\label{eq:1000}
  \int_{\mathbb R^d} |V| |\psi|^2
  \leq
  \|r V\psi\|\left\|\frac{\psi}{r}\right\|
  \leq
  \frac{2b}{d-2}\int_{\mathbb R^d}|\nabla\psi|^2
  \,.
\end{equation}
for all $\psi\in H^1(\mathbb R^d)$,
and $b<(d-2)/2$ due to the restriction in~\eqref{Ass.bis}.

Both~\eqref{Ass} and~\eqref{Ass.bis} are smallness assumptions about~$V$.
Our next step is to look for some alternative conditions which guarantee 
the absence of eigenvalues for~$H_V$, 
in all dimensions $d \geq 3$. 
The idea is to modify 
the proof of Theorem~\ref{Thm.any} by splitting the real and imaginary parts 
of the potential~$V$ and treating them separately.
In order to include potentials which are not necessarily 
subordinated in the spirit of~\eqref{Ass}, we consider the space
\begin{equation}\label{space}
  \nspace(\Real^d) := \overline{C_0^\infty(\Real^d)}^{\vertiii{\cdot}} 
  \,, \qquad
  \vertiii{\psi}^2 := \int_{\Real^d} |\nabla\psi|^2
  + \int_{\Real^d} (\Re V)_+ \, |\psi|^2
  + \int_{\Real^d} |\psi|^2
  \,,
\end{equation}
where we have introduced the notation $f_\pm := \max\{\pm f,0\}$
for any measurable function $f:\Real^d \to \Real$.
Clearly, $\nspace(\Real^d)$ is continuously embedded in $H^1(\Real^d)$
and it coincides with the latter as a set if~\eqref{Ass} holds. 
The form 
$
  h_V^{(1)}[\psi] := \int_{\Real^d} |\nabla\psi|^2
  + \int_{\Real^d} (\Re V)_+ \;\! |\psi|^2
$,
$\Dom(h_V^{(1)}):=\nspace(\Real^d)$, is closed by definition.
Assuming now only that $(\Re V)_-$ and $\Im V$ are form-subordinated
to~$H_0$ with the subordination bound less than one
(\cf~\eqref{eq:assV12} and~\eqref{eq:assV2} below),
the sum $h_V := h_V^{(1)} + h_V^{(2)}$
with
$
  h_V^{(2)}[\psi] :=  -\int_{\Real^d} (\Re V)_- \;\! |\psi|^2
  + \int_{\Real^d} \Im V \;\! |\psi|^2
$
is a closed form with $\Dom(h_V)=\nspace(\Real^d)$.
Of course, $h_V$~coincides with the previously defined form
under the hypothesis~\eqref{Ass}. 
In this more general setting, we also denote by~$H_V$
the m-sectorial operator associated with~$h_V$.

Now we are in a position to state the main result about
the absence of eigenvalues for~$H_V$ under natural 
conditions on~$V$.
\begin{Theorem}\label{thm:radi}
Let $d\geq3$ and assume that there exist non-negative numbers $b_1,b_2,b_3$ 
satisfying 
\begin{equation}\label{eq:assf.Lambdanew}
  b_1^2<1-\frac{2b_3}{d-2} 
  \,,
  \qquad
  b_2^2 + 2 \, b_3 
  + \frac{1}{4} \sqrt{b_3} \, \left(\frac{2}{d-2}\right)^{\frac32}
  < 1
  \,,
\end{equation}
such that, 
for all $\psi\in \nspace(\mathbb R^d)$,
  \begin{align}
  \label{eq:assV12}
  \int_{\mathbb R^d} (\Re V)_- \, |\psi|^2 
  &
  \leq
  b_1^2\int_{\mathbb R^d}|\nabla\psi|^2 \,,
  \\
  \label{eq:asspart}
  \int_{\mathbb R^d}\left[\partial_r(r \, \Re V)\right]_+ \, |\psi|^2
    &
    \leq
   b_2^2\int_{\mathbb R^d}|\nabla\psi|^2 \,,
   \\
   \int_{\mathbb R^d} r^2 \, |\Im V|^2 \, |\psi|^2 
  &
  \leq
  b_3^2\int_{\mathbb R^d}|\nabla\psi|^2 \,,
  \label{eq:assV2}
  \end{align}
where 
$\partial_r f(x) := \frac{x}{|x|} \cdot \nabla f(x)$.
 Then 
$
  \sigma_\mathrm{p}(H_V) = \varnothing
$.
\end{Theorem}

We recall that~\eqref{eq:assV12} and~\eqref{eq:assV2} ensure that 
$h_V^{(2)}$ is subordinated to $h_V^{(1)}$ 
with the subordination bound less than one,
so~$H_V$ is indeed well defined.
A brief comparison between Theorems~\ref{Thm.main}, \ref{Thm.any} and~\ref{thm:radi} 
is in order:
\begin{itemize}
\item
If $\Im V=0$, namely~$V$ is real-valued, 
then~$b_3$ can be chosen to be equal to zero
and condition~\eqref{eq:assf.Lambdanew} then reads $b_1<1$, $b_2<1$. 
In this case, the subordination assumption~\eqref{Ass} implies~\eqref{eq:assV12}. 
However, we stress that conditions~\eqref{eq:assV12} and~\eqref{eq:asspart} 
are not unsigned, contrary to the case of~\eqref{Ass}. 
In particular, a large class of {\it repulsive} potentials 
such as the Coulomb-type interaction $V(x) = c \;\! |x|^{-1}$ with any $c>0$ 
satisfy~\eqref{eq:assV12} and~\eqref{eq:asspart}, 
although the subordination~\eqref{Ass} fails. 

\item
On the other hand, if $\Re V=0$, namely~$V$ is purely imaginary-valued, 
then \eqref{eq:assV12}, \eqref{eq:asspart} are fulfilled 
and one just needs to assume~\eqref{eq:assV2} with
$$
  \sqrt{b_3} < 8 
  \left[
  \left(\frac{2}{d-2}\right)^{\frac32}
  + \sqrt{\left(\frac{2}{d-2}\right)^{3}+128}
  \right]^{-1}
  .
$$
This hypothesis is better than condition~\eqref{Ass.bis} of Theorem~\ref{Thm.any} 
and represents a completely new result, to our knowledge. 
However, for general complex-valued potentials~$V$,
the interest of Theorem~\ref{Thm.any} consists in that 
it requires no conditions on the derivatives of~$V$.
\end{itemize}

The techniques used to prove Theorems~\ref{Thm.any} and~\ref{thm:radi} 
permit to handle more general lower-order perturbations of~$H_0$. 
It is of particular interest for the \emph{electromagnetic Hamiltonian}~$H_{A,V}$
that we introduce as follows. 
Given a \emph{magnetic potential} $A\in L^2_{\text{loc}}(\mathbb R^d;\mathbb R^d)$
and denoting by $\nabla_{\!A} := \nabla+iA$ the magnetic gradient,
we now consider the space 
\begin{equation}\label{Aspace}
  \nspace_{\!A}(\Real^d) := \overline{C_0^\infty(\Real^d)}^{\vertiii{\cdot}_A} 
  \,, \qquad
  \vertiii{\psi}_A^2 := \int_{\Real^d} |\nabla_{\!A}\psi|^2
  + \int_{\Real^d} (\Re V)_+ \, |\psi|^2 + \int_{\Real^d} |\psi|^2
  \,,
\end{equation}
and introduce the form
$
  h_{A,V}[\psi] := \int_{\Real^d} |\nabla_{\!A}\psi|^2
  + \int_{\Real^d} V |\psi|^2
$,
$\Dom(h_{A,V}):=\nspace_{\!A}(\Real^d)$.
If~$V$ is such that~\eqref{eq:assV1bis} and~\eqref{eq:assV2bis} below hold,
then~$h_{A,V}$ is closed. 
We denote by~$H_{A,V}$ the m-sectorial operator 
associated with~$h_{A,V}$.
We next denote by $B:=\nabla A-(\nabla A)^t\in\mathcal M_{d\times d}(\mathbb R)$ 
the \emph{magnetic field} generated by~$A$. 
(For $d=3$, $B$~may be identified with $\curl A$, in the sense that
$Bv=\curl A \times v$ for all $v\in\mathbb R^3$,
where the cross denotes the vectorial product.)
Following a notation introduced in~\cite{fanelli_vega},
we also define
\begin{equation}\label{eq:bitau}
  B_\tau(x)
  :=
  \frac{x}{|x|}\cdot B(x)
  \,.
\end{equation}
(A non-trivial example of magnetic field with $B_\tau=0$ 
is given in dimension $d=3$ by the magnetic potential
$
A(x) = |x|^{-2}(-x_2,x_1,0).
$)

The last result of this manuscript is an analogue of Theorem~\ref{thm:radi}
in the presence of an external magnetic field.

\begin{Theorem}\label{thm:radimagn}
  Let $d\geq3$, $A\in L^2_{\mathrm{loc}}(\mathbb R^d;\mathbb R^d)$
and assume that there exist non-negative numbers $b_1,b_2,b_3$ 
satisfying~\eqref{eq:assf.Lambdanew} such that, 
for all $\psi\in \nspace_{\!A}(\mathbb R^d)$,
  \begin{align}\label{eq:assV1bis}
  \int_{\mathbb R^d} (\Re V)_- \, |\psi|^2
  &
  \leq
  b_1^2\int_{\mathbb R^d}|\nabla_{\!A}\psi|^2 \,,
  \\
  \label{eq:asspartbis}
  \int_{\mathbb R^d}\left[\partial_r(r \, \Re V)\right]_+ \, |\psi|^2
    &
    \leq
   b_2^2\int_{\mathbb R^d}|\nabla_{\!A}\psi|^2 \,,
   \\
   \int_{\mathbb R^d} r^2 \left(|\Im V|^2+\frac12|B_\tau|^2\right) \, |\psi|^2
  &
  \leq
  b_3^2\int_{\mathbb R^d}|\nabla_{\!A}\psi|^2
  \label{eq:assV2bis}
  \,.
  \end{align}
Then 
$
  \sigma_\mathrm{p}(H_{A,V}) = \varnothing
$.
\end{Theorem}

\section{Absence of eigenvalues: the Birman-Schwinger principle}\label{Sec.evs}
The main role in our proof of Theorem~\ref{Thm.main}
is played by the Birman-Schwinger operator
\begin{equation*}
  K_z := |V|^{1/2} \, (H_0-z)^{-1} \, V_{1/2}
  \qquad \mbox{with} \qquad
  V_{1/2} := |V|^{1/2} \, \sgn(V) 
  \,,
\end{equation*}
where $\sgn(z)$ is the complex signum function 
defined by $\sgn(z):=z/|z|$ for $z \in \Com\setminus\{0\}$
and $\sgn(0):=0$.
We abuse the notation by using the same symbols 
for maximal operators of multiplication
and their generating functions.
The operator~$K_z$ is well defined 
(on its natural domain of the composition of three operators)
for all $z \in \Com$ and $d \geq 3$.

If $z \not \in [0,+\infty)$, however,
we have a useful formula for the integral kernel of~$K_z$:
\begin{equation}\label{K.op}
  K_z(x,y) = |V|^{1/2}(x) \, G_z(x,y) \, V_{1/2}(y)
  \,,
\end{equation}
where~$G_z$ is the \emph{Green's function} of~$H_0-z$,
\ie~the integral kernel of the resolvent~$(H_0-z)^{-1}$. 
We observe that~$K_z$ is a bounded operator 
for all $z \not \in [0,+\infty)$ and $d \geq 3$
under our hypothesis~\eqref{Ass}. 
Indeed, 
$V_{1/2}$ maps $\sii(\Real^d)$ to $H^{-1}(\Real^d)$ by duality,
$(H_0-z)^{-1}$ is an isomorphism between $H^{-1}(\Real^d)$ and $H^1(\Real^d)$
and the latter space is mapped by~$|V|^{1/2}$ back to $\sii(\Real^d)$.     

Moreover, if $d=3$, we have an explicit formula
\begin{equation}\label{Green}
  G_z(x,y) := \frac{1}{4\pi} \frac{e^{-\sqrt{-z}\,|x-y|}}{|x-y|}
  \,.
\end{equation}
Here and in the sequel we choose the principal branch of the square root.
Using this explicit formula,
we are able to show that $K_z$~is bounded by~$a$ 
under the hypothesis~\eqref{Ass}.
\begin{Lemma}\label{Lem.bound}
Let $d=3$ and assume~\eqref{Ass}. Then 
\begin{equation}\label{bound} 
  \forall z \not\in (0,+\infty)
  \,, \qquad
  \|K_z\| \leq a 
  \,.
\end{equation}
\end{Lemma}
\begin{proof}
We start with an equivalent formulation of~\eqref{Ass}, 
in any dimension $d \geq 3$. 
Writing $g := H_0^{1/2} \psi$ in~\eqref{Ass}, we have
\begin{equation*} 
  \big\||V|^{1/2}H_0^{-1/2} g\big\|^2 
  \leq a \, \big\|\nabla H_0^{-1/2} g\big\|^2  
  = a \, \|g\|^2
  \,,
\end{equation*}
where~$\|\cdot\|$ denotes the norm in $\sii(\Real^d)$.
Since the range of~$H_0^{1/2}$ is dense in $\sii(\Real^d)$,
we see that~\eqref{Ass} is equivalent to
\begin{equation}\label{b1} 
  \big\||V|^{1/2}H_0^{-1/2}\big\|^2 
  \leq a  
  \,.
\end{equation}
It follows (by taking the adjoint) that also
\begin{equation}\label{b2}
  \big\|H_0^{-1/2}|V|^{1/2}\big\|^2 
  \leq a  
  \,.
\end{equation}

Now we assume $d=3$, where the explicit formula~\eqref{Green}
for the Green function is available.
By virtue of the pointwise bound
\begin{equation}\label{crucial} 
  \forall z \not\in (0,+\infty)
  \,, \quad 
  \forall x,y \in \Real^3
  \,, \qquad 
  |G_z(x,y)| \leq G_0(x,y)
  \,,
\end{equation}
we have 
\begin{equation}\label{crucial.consequence} 
  |(f,K_z g)| \leq (|f|,\tilde{K}_0|g|) 
  \leq \|\tilde{K}_0\| \|f\| \|g\|
\end{equation}
for every $z \not\in (0,+\infty)$ and all $f,g \in \sii(\Real^3)$,
where
\begin{equation*}
  \tilde{K}_0 := |V|^{1/2} H_0^{-1} |V|^{1/2}
\end{equation*}
and $(\cdot,\cdot)$ denotes the inner product in $\sii(\Real^3)$
(conjugate linear in the first argument).
Using~\eqref{b1} and~\eqref{b2}, we have
\begin{equation}\label{tilde.crucial.consequence}
  \|\tilde{K}_0\| = \big\||V|^{1/2}  H_0^{-1} |V|^{1/2}\big\|
  \leq \big\||V|^{1/2} H_0^{-1/2}\big\| \big\|H_0^{-1/2} |V|^{1/2}\big\|
  \leq a \,.
\end{equation}
Consequently, \eqref{crucial.consequence} and \eqref{tilde.crucial.consequence}
imply~\eqref{bound}.
\end{proof}

The following lemma provides an (integral) criterion 
for the existence of solutions to the (differential) eigenvalue 
equation of~$H_V$. It can be considered as a one-sided version
of the Birman-Schwinger principle extended to possible eigenvalues 
in $[0,+\infty)$ as well.
\begin{Lemma}\label{Lem.BS}
Let $d=3$ and assume~\eqref{Ass}. 
If $H_V\psi=\lambda\psi$ with some $\lambda \in \Com$ and $\psi\in\Dom(H_V)$, 
then $\phi := |V|^{1/2}\psi$ obeys
\begin{equation}\label{BS}
  \forall \varphi \in \sii(\Real^3) 
  \,, \qquad
  \lim_{\eps \to 0^\pm} (\varphi,K_{\lambda + i\eps} \phi) = - (\varphi,\phi) 
  \,.
\end{equation}
\end{Lemma}
\begin{proof}
Given any $\lambda \in \Com$, there is $\eps_0 > 0$ such that 
$\lambda + i\eps \not \in [0,+\infty)$ for all real~$\eps$
satisfying $0 < |\eps| < \eps_0$.  
By density of $C_0^\infty(\Real^3)$ in $\sii(\Real^3)$ and Lemma~\ref{Lem.bound},
it is enough to prove~\eqref{BS} for $\varphi \in C_0^\infty(\Real^3)$.
We have 
\begin{equation}\label{i1}
  (\varphi,K_{\lambda + i\eps} \phi)
  = \iint_{\Real^3\times\Real^3} 
  \overline{\varphi(x)} \, |V|^{1/2}(x) \, 
  G_{{\lambda + i\eps}}(x,y) \, V(y) \, \psi(y) 
  \, dx \, dy
  = \int_{\Real^3} \eta_\eps(y) \, V(y) \, \psi(y) \, dy
  \,,
\end{equation}
where
\begin{equation*}
  \eta_\eps := \int_{\Real^3} 
  \overline{\varphi(x)} \, |V|^{1/2}(x) \, G_{{\lambda + i\eps}}(x,\cdot) \, dx
  = (H_0-\lambda-i\eps)^{-1} \, |V|^{1/2} \, \overline{\varphi}
  \,,
\end{equation*}
where the second equality holds due to the symmetry $G_z(x,y)=G_z(y,x)$.
In view of~\eqref{Ass}, $|V|^{1/2} \overline{\varphi} \in \sii(\Real^3)$.
Since $\eps \not=0$ is so small that $\lambda + i\eps \not\in \sigma(H_0)$,
we have $\eta_\eps \in \Dom(H_0) = H^2(\Real^3)$.
In particular, $\eta_\eps \in H^1(\Real^3)$ and the weak formulation
of the eigenvalue equation $H_V\psi=\lambda\psi$ yields
\begin{equation}\label{i2}
\begin{aligned}
  \int_{\Real^3} \eta_\eps(y) \, V(y) \, \psi(y) \, dy
  &= -(\nabla\overline{\eta_\eps},\nabla\psi) 
  + \lambda \, (\overline{\eta_\eps},\psi)
  \\
  &= -(\nabla\overline{\psi},\nabla\eta_\eps) 
  + \lambda \, (\overline{\psi},\eta_\eps)
  \\
  &= -(\nabla\overline{\psi},\nabla\eta_\eps) 
  + (\lambda+i\eps) \, (\overline{\psi},\eta_\eps)
  - i\eps \, (\overline{\psi},\eta_\eps)
  \\
  &= -(\overline{\psi},|V|^{1/2}\overline{\varphi}) 
  - i\eps \, (\overline{\psi},\eta_\eps)
  \\
  &= -(\varphi,|V|^{1/2}\psi) 
  - i\eps \, (\overline{\eta_\eps},\psi)
  \,.
\end{aligned}
\end{equation}
Here the last but one equality follows from the weak
formulation of the resolvent equation
$
  (H_0-\lambda-i\eps)\eta_\eps = |V|^{1/2} \overline{\varphi}
$.
Consequently, \eqref{i1} and~\eqref{i2} imply~\eqref{BS}
after taking the limit $\eps \to 0^\pm$, 
provided that $\eps \, (\bar\eta_\varepsilon,\psi) \to 0$
as $\eps \to 0$.
To see the latter, we write
\begin{equation*}
  |(\overline{\eta_\eps},\psi)| 
  = |(\varphi,M_\eps\psi)|
  \leq \|\varphi\| \;\! \|M_\eps\| \;\! \|\psi\|
  \,,
\end{equation*}
where
$
  M_\eps := \chi_\Omega \, |V|^{1/2} (H_0-\lambda-i\eps)^{-1}
$
with $\Omega := \supp\varphi$,
and it remains to show that $\eps \, \|M_\eps\|$ tends to zero as $\eps \to 0$.
Following~\cite[Thm.~III.6]{SiQF}, 
we use the resolvent kernel~\eqref{Green} 
and estimate $\|M_\eps\| \leq \|M_\eps\|_\mathrm{HS}$.
We have
$$
  \|M_\eps\|_\mathrm{HS}^2
  = \frac{1}{(4\pi)^2} \iint_{\Omega \times \Real^3} 
  |V(x)| \, \frac{e^{-2 \;\! \kappa(\eps) \;\! |x-y|}}{|x-y|^2}
  \, dx \, dy
  = \frac{1}{4\pi \kappa(\eps)} \int_\Omega |V(x)| \, dx
  \,,
$$
where the last integral is bounded because $V \in L_\mathrm{loc}^1(\Real^3)$
as a consequence of~\eqref{Ass} and  
$$
  \kappa(\eps) := \Re \sqrt{-(\lambda+i\eps)}
  \sim
  \begin{cases}
    |\eps|^{1/2} & \mbox{if} \quad \lambda = 0 \,,
    \\
    |\eps| & \mbox{if} \quad \Re\lambda > 0 \ \& \ \Im\lambda = 0 \,,
    \\
    1 & \mbox{otherwise} \,.
  \end{cases}
$$
Hence, $\eps \, \|M_\eps\|$ behaves at least as $\mathcal{O}(\eps^{1/2})$
as $\eps \to 0$, which concludes the proof of the lemma.
\end{proof}
\begin{Remark}
Lemma~\ref{Lem.BS} resembles~\cite[Thm.~III.6]{SiQF} in the self-adjoint case. 
It is also related to the recent abstract result \cite[Prop.~3.1]{Frank-Simon}.
\end{Remark}

Now we are in a position to establish the absence of eigenvalues
in three dimensions.
\begin{Theorem}\label{Thm.evs}
Let $d=3$ and assume~\eqref{Ass}. Then 
$
  \sigma_\mathrm{p}(H_V) 
  = \varnothing
$.
\end{Theorem}
\begin{proof}
Assume there exists $\lambda \in \Com$ and a non-trivial $\psi \in \Dom(H_V)$
such that $H_V\psi=\lambda\psi$. 
Since the spectrum of~$H_0$ is purely continuous, 
the theorem clearly holds for~$V=0$ 
and we may thus suppose that~$V$ is non-trivial.    
But then $\phi := |V|^{1/2}\psi$ is also non-trivial,
otherwise~$\psi$ would be a non-trivial solution 
of $H_0\psi=\lambda\psi$, which is again impossible
by the absence of eigenvalues for~$H_0$.   
Now, Lemma~\ref{Lem.BS} with $\varphi:=\phi$ and Lemma~\ref{Lem.bound} yield
\begin{equation}
  a \, \|\phi\| ^2\geq 
  \lim_{\eps \to 0^\pm} |(\phi,K_{\lambda + i\eps} \phi)| 
  = \|\phi\|^2  
  \,.
\end{equation}
This is a contradiction because $a < 1$.
\end{proof}

\section{Absence of the continuous spectrum outside 
\texorpdfstring{$[0,+\infty)$}{positive}}\label{Sec.cont1}
The following lemma is a modification of the idea behind Lemma~\ref{Lem.BS}
to deal with the continuous spectrum.
We prove it in all dimensions $d \geq 3$.
\begin{Lemma}\label{Lem.BS.cont}
Let $d \geq 3$ and assume~\eqref{Ass}. 
If $\|H_V\psi_n-\lambda\psi_n\| \to 0$ as $n \to \infty$ 
with some $\lambda \in \Com \setminus \Real$ 
and $\{\psi_n\}_{n\in\Nat}\subset\Dom(H_V)$
such that $\|\psi_n\|=1$ for all $n \in \Nat$, 
then $\phi_n := |V|^{1/2}\psi_n$ obeys
\begin{equation}\label{BS.cont}
  \lim_{n \to \infty} 
  \frac{(\phi_n,K_{\lambda} \phi_n)}{\|\phi_n\|^2} = - 1 
  \,.
\end{equation}
\end{Lemma}
\begin{proof}
The proof is similar to that of Lemma~\ref{Lem.BS}.
We have 
\begin{equation}\label{i1.cont}
  (\phi_n,K_{\lambda} \phi_n)
  = \int_{\Real^d} \eta_n(y) \, V(y) \, \psi_n(y) \, dy
  = v(\overline{\eta_n},\psi_n)
  \,,
\end{equation}
where $(\cdot,\cdot)$ denotes the inner product in $\sii(\Real^d)$
and the function
\begin{equation*}
  \eta_n := \int_{\Real^d} 
  \overline{\phi_n(x)} \, |V|^{1/2}(x) \, G_{{\lambda}}(x,\cdot) \, dx
  = (H_0-\lambda)^{-1} \, |V|^{1/2} \, \overline{\phi_n}
\end{equation*}
belongs to $H^1(\Real^d)$.
Indeed,
\begin{equation}\label{decomposition}
  \eta_n 
  = (H_0-\lambda)^{-1} \, H_0^{1/2} 
  H_0^{-1/2}\, |V|^{1/2} \, \overline{\phi_n}
  \,,
\end{equation}
where $\phi_n \in \sii(\Real^d)$ by~\eqref{Ass},
$H_0^{-1/2} |V|^{1/2}$ is bounded due to~\eqref{b2}
and $(H_0-\lambda)^{-1} H_0^{1/2}$ maps $\sii(\Real^d)$ to $H^1(\Real^d)$.
More specifically,
\begin{equation}\label{decomposition.bound}
  \|\eta_n\| \leq C_\lambda \, \sqrt{a} \, \|\phi_n\| 
  \,, \qquad \mbox{where} \qquad
  C_\lambda := \sup_{\xi\in[0,\infty)} \left| \frac{\xi}{\xi^2-\lambda} \right|
  \,.
\end{equation}
In analogy with~\eqref{i2}, we are thus allowed to write
\begin{equation}\label{i2.cont}
\begin{aligned}
  v(\overline{\eta_n},\psi_n)
  &= h_V(\overline{\eta_n},\psi_n) - \lambda \, (\overline{\eta_n},\psi_n)
  - (\nabla\overline{\eta_n},\nabla\psi_n) + \lambda \, (\overline{\eta_n},\psi_n)
  \\
  &= \big(\overline{\eta_n},(H_V-\lambda)\psi_n\big)
  - h_0(\overline{\psi_n},\eta_n) + \lambda \, (\overline{\psi_n},\eta_n)
  \,.
\end{aligned}
\end{equation}
By the second representation theorem (\cf~\cite[Thm.~VI.2.23]{Kato})
and~\eqref{decomposition},
\begin{equation}\label{i2.cont.bis}
\begin{aligned}
  h_0(\overline{\psi_n},\eta_n) - \lambda \, (\overline{\psi_n},\eta_n)
  &= \big(H_0^{1/2}\overline{\psi_n},H_0^{1/2}\eta_n\big)
  - \lambda \, (\overline{\psi_n},\eta_n)
  \\
  &= \big(H_0^{1/2}\overline{\psi_n},
  (H_0-\lambda+\lambda) (H_0-\lambda)^{-1} H_0^{-1/2}|V|^{1/2}\overline{\phi_n}\big)
  - \lambda \, (\overline{\psi_n},\eta_n)
  \\
  &= \big(H_0^{1/2}\overline{\psi_n},
  H_0^{-1/2}|V|^{1/2}\overline{\phi_n}\big)
  \\
  &= \big( (H_0^{-1/2}|V|^{1/2})^*H_0^{1/2}\overline{\psi_n},
  \overline{\phi_n}\big)
  \\
  &= \big( |V|^{1/2} \overline{\psi_n},
  \overline{\phi_n}\big)
  \\
  &= \|\phi_n\|^2
  \,.
\end{aligned}
\end{equation}

Since 
$$
  \|H_V\psi_n-\lambda\psi_n\|
  = \sup_{\stackrel[\varphi\not=0]{}{\varphi \in \sii(\Real^d)}}
  \frac{|(\varphi,H_V\psi_n-\lambda\psi_n)|}{\|\varphi\|}
  \geq \left| \|\nabla\psi_n\|^2 + v[\psi_n] - \lambda \right|
  \,,
$$
where the inequality is obtained by choosing $\varphi:=\psi_n$,
and the left hand side vanishes as $n \to \infty$,
we see that $\Im v[\psi_n]$ tends to $\Im \lambda \not= 0$ as $n \to \infty$.
In particular, 
\begin{equation}\label{liminf}
  \liminf_{n \to \infty} \|\phi_n\| > 0
  \,.
\end{equation}

Using~\eqref{i2.cont.bis} in~\eqref{i2.cont}, recalling~\eqref{i1.cont},
dividing the obtained identity by~$\|\phi_n\|^2$
(which is non-zero for all sufficiently large~$n$ due to~\eqref{liminf})
and taking the limit as $n \to \infty$, we arrive at
$$
  \lim_{n \to \infty} 
  \frac{(\phi_n,K_{\lambda} \phi_n)}{\|\phi_n\|^2} + 1 
  = \lim_{n \to \infty} 
  \frac{\big(\overline{\eta_n},(H_V-\lambda)\psi_n\big)}{\|\phi_n\|^2}
  \,.
$$
In view of~\eqref{decomposition.bound} and~\eqref{liminf},
the right hand side equals zero by the hypothesis.
\end{proof}

Now we are in a position to establish the absence of 
the continuous spectrum outside $[0,+\infty)$.
\begin{Theorem}\label{Thm.cont1}
Let $d = 3$ and assume~\eqref{Ass}. Then 
$
  \sigma_\mathrm{c}(H_V) 
  \subset [0,+\infty)
$.
\end{Theorem}
\begin{proof}
By~\eqref{Ass}, $\Re h_V[\psi] \geq (1-a) \|\nabla\psi\|^2 \geq 0$
for all $\psi \in H^1(\Real^3)$. 
Since~$H_V$ is m-sectorial, it follows that the spectrum of~$H_V$
is contained in the right complex half-plane
(\cf~\cite[Thm.~V.3.2]{Kato}).
Assume that there exists
$\lambda \in \Com$ with $\Re\lambda \geq 0$ and $\Im\lambda \not= 0$
such that $\lambda \in \sigma_\mathrm{c}(H_V)$.
Then~$\lambda$ belongs to the kind of essential spectrum 
which is characterised by the existence of a singular sequence of~$H_V$
corresponding to~$\lambda$ (\cf~\cite[Thm.~IX.1.3]{Edmunds-Evans}):  
$\exists \{\psi_n\}_{n \in \Nat} \subset \Dom(H_V)$, 
$\|\psi_n\|=1$ for all $n \in \Nat$, 
$\|(H_V-\lambda)\psi_n\| \to 0$ as $n \to \infty$
and $\{\psi_n\}_{n \in \Nat}$ is weakly converging to zero.
By Lemma~\ref{Lem.BS.cont} and Lemma~\ref{Lem.bound},
\begin{equation*} 
  a
  \geq 
  \|K_\lambda\|
  \geq
  \left| \lim_{n \to \infty} 
  \frac{(\phi_n,K_{\lambda} \phi_n)}{\|\phi_n\|^2} 
  \right|
  =  1 
  \,,
\end{equation*}
This is a contradiction because $a < 1$.
\end{proof}

We remark that the last step of the proof of Theorem~\ref{Thm.cont1}
requires Lemma~\ref{Lem.bound} for which $d=3$ is crucial.

\section{Inclusion of the spectrum in 
\texorpdfstring{$[0,+\infty)$}{positive}}\label{Sec.cont2}
The opposite inclusion follows by an explicit 
construction of a singular sequence of~$H_V$
corresponding to non-negative energies.
Since the operator~$H_V$ is defined through its sesquilinear form,
it is convenient to have a criterion which requires that
the singular sequence is in the form domain only.
Unable to find a reference in the general case,
we state the abstract version first
(for the self-adjoint situation, see \cite[Thm.~5]{KL}).
\begin{Lemma}\label{Lem.criterion}
Let~$H$ be an m-sectorial accretive operator in a complex Hilbert space~$\mathcal{H}$
which is associated with a (densely defined, closed, sectorial) 
sesquilinear form~$h$. Given $\lambda \in \Com$, assume that there exists
a sequence $\{\phi_n\}_{n\in\Nat} \subset \Dom(h)$ 
such that $\|\phi_n\|=1$ for all $n\in\Nat$ and
\begin{equation}\label{criterion} 
  \sup_{\stackrel[\psi\not=0]{}{\psi \in \Dom(h)}}
  \frac{|h(\phi_n,\psi) - \lambda \, (\phi_n,\psi)|}{\|\psi\|_{\Dom(h)}}
  \xrightarrow[n \to \infty]{}
  0 \,,
\end{equation}
where $\|\psi\|_{\Dom(h)} := \sqrt{\Re h[\psi]+\|\psi\|^2}$.
Then $\lambda \in \sigma(H)$.
\end{Lemma}
\begin{Remark}
Notice that the left hand side of~\eqref{criterion}
is the norm of the vector $H^*\phi_n-\overline{\lambda}\;\!\phi_n$ 
in the dual space $\Dom(h)^*$, 
when $\Dom(h)$ is thought as the subspace of~$\mathcal{H}$
equipped with the norm $\|\cdot\|_{\Dom(h)}$.
\end{Remark}
\begin{proof}
We proceed by contradiction: 
Assume the hypotheses of the theorem and $\lambda \not\in \sigma(H)$.
The latter means that for every $g \in \mathcal{Hilbert}$ 
there exists $\psi \in \Dom(H)$ such that $H\psi - \lambda\psi = g$.
That is, $\psi = (H-\lambda)^{-1} g$ and $(H-\lambda)^{-1}$ is bounded
as an operator on~$\mathcal{Hilbert}$ onto~$\mathcal{Hilbert}$. 
The weak formulation of the resolvent equation reads
\begin{equation}\label{res.eq}
  \forall \phi \in \Dom(h)
  \,, \qquad
  h(\phi,\psi) - \lambda \, (\phi,\psi) = (\phi,g)
  \,.
\end{equation}
Consequently, for every $\phi \in \Dom(h)$,
\begin{equation}\label{criterion.bound}
  C_\lambda \,
  \sup_{\stackrel[\psi\not=0]{}{\psi \in \Dom(h)}}
  \frac{| h(\phi,\psi) - \lambda \, (\phi,\psi) |}{\|\psi\|_{\Dom(h)}}
  \geq
  \sup_{\stackrel[g\not=0]{}{g \in \mathcal{Hilbert}}}
  \frac{| h(\phi,\psi) - \lambda \, (\phi,\psi) |}{\|g\|}
  = \sup_{\stackrel[g\not=0]{}{g \in \mathcal{Hilbert}}}
  \frac{|(\phi,g)|}{\|g\|}
  = \|\phi\|
  \,,
\end{equation}
where~$\psi$ and~$g$ are related through~\eqref{res.eq}
and the constant
$$
  C_\lambda := \sup_{\stackrel[g\not=0]{}{g \in \mathcal{Hilbert}}}
  \frac{\|\psi\|_{\Dom(h)}}{\|g\|}
$$
is finite because the resolvent $(H-\lambda)^{-1}$ 
maps~$\mathcal{Hilbert}$ onto $\Dom(H) \subset \Dom(h)$.
More specifically,
$$
\begin{aligned}
  \|\psi\|_{\Dom(h)}^2
  &= \Re h[(H-\lambda)^{-1}g] + \|(H-\lambda)^{-1}g\|^2  
  \\
  &= \Re \big((H-\lambda)^{-1}g,H(H-\lambda)^{-1}g\big)
  + \|(H-\lambda)^{-1}g\|^2  
  \\
  &\leq \left(
  \|(H-\lambda)^{-1}\| \|H(H-\lambda)^{-1}\|
  + \|(H-\lambda)^{-1}\|^2
  \right) \|g\|^2
  \,.
\end{aligned}
$$
Choosing $\phi := \phi_n$ in~\eqref{criterion.bound},
we get that the left hand side tends to zero as $n \to \infty$ 
by~\eqref{criterion},
while the right hand side equals one due to the normalisation 
of $\{\phi_n\}_{n\in\Nat}$, a contradiction.
\end{proof}

Now we are in a position to prove the inclusion 
of the interval $[0,+\infty)$ in the spectrum of~$H_V$.
The following result holds in all dimensions $d \geq 3$.
\begin{Theorem}\label{Thm.cont2}
Let $d \geq 3$ and assume~\eqref{Ass}. Then 
$
  \sigma(H_V) 
  \supset [0,+\infty)
$.
\end{Theorem}
\begin{proof}
We construct the sequence $\{\phi_n\}_{n\in\Nat}$
from Lemma~\ref{Lem.criterion} applied to~$H_V$ by setting
$$
  \phi_n(x) := \varphi_n(x) \, e^{i k\cdot x}
  \,, 
$$
where $k \in \Real^d$ is such that 
$|k|^2 = \lambda \in [0,+\infty)$,
$\varphi_n(x) := n^{-d/2} \varphi_1(x/n)$ 
for all $n \in \Nat$ (with the convention $0 \not\in \Nat$)
and $\varphi_1 \in C_0^\infty(\Real^d)$ is a function
such that $\|\varphi_1\|=1$.
The normalisation factor in the definition of~$\varphi_n$
is chosen in such a way that 
$$
  \|\varphi_n\| = \|\varphi_1\| = 1 
  \,, \qquad
  \|\nabla\varphi_n\| = n^{-1} \, \|\nabla\varphi_1\|
  \,, \qquad
  \|\Delta\varphi_n\| = n^{-2} \, \|\Delta\varphi_1\|
$$
for all $n \in \Nat$.
Then $\|\phi_n\| = 1$ and $\phi_n \in \Dom(h_V) = \Dom(h_0) = H^1(\Real^d)$ 
for all $n \in \Nat$. Furthermore,
\begin{equation}\label{Weyl1}
  \left\| -\Delta\phi_n - \lambda\;\!\phi_n \right\|
  = \left\|
  -\Delta\varphi_n + 2 i k \cdot \nabla \varphi_n 
  \right\|
  \leq \|\Delta\varphi_n\| + 2 \, |k| \;\! \|\nabla\varphi_n\| 
  \xrightarrow[n \to \infty]{}
  0 \,.
\end{equation}
In fact, $\{\phi_n\}_{n\in\Nat}$ is the usual singular sequence of~$H_0$
corresponding to~$\lambda$.
At the same time,
\begin{equation}\label{Weyl2}
  \big| v[\phi_n] \big|
  \leq \big\| |V|^{1/2} \varphi_n \big\|^2
  \leq a \, \|\nabla \varphi_n\|^2
  \xrightarrow[n \to \infty]{}
  0 \,,
\end{equation}
where the second inequality follows by~\eqref{Ass}.

The numerator in~\eqref{criterion} can be estimated as follows
$$
\begin{aligned}
  |h_V(\phi_n,\psi) - \lambda \, (\phi_n,\psi)|
  &= |(-\Delta\phi_n -\lambda\phi_n,\psi) + v(\phi_n,\psi)|
  \\
  &\leq \left\|-\Delta\phi_n -\lambda\;\!\phi_n\right\| \|\psi\|
  + \sqrt{\left|v[\phi_n]\right|} \sqrt{\left|v[\psi]\right|} 
  \\
  & \leq \left\|-\Delta\phi_n -\lambda\;\!\phi_n\right\| \|\psi\|
  + \sqrt{\left|v[\phi_n]\right|} \, \sqrt{a} \, \|\nabla\psi\|
  \,,
\end{aligned}
$$
where the last inequality is due to~\eqref{Ass}.
As for the denominator in~\eqref{criterion}, 
employing~\eqref{Ass} again, we have 
$$
  \|\psi\|_{\Dom(h)}^2 
  = \|\nabla\psi\|^2 + \Re v[\psi] + \|\psi\|^2
  \geq (1-a) \, \|\nabla\psi\|^2 + \|\psi\|^2
  \geq (1-a) \, \|\psi\|_{\Dom(h_0)}^2 
  \,,
$$
where $\|\cdot\|_{\Dom(h_0)}$ is just the usual norm of $H^1(\Real^d)$.
Putting these estimates together, we have the bound
$$
  \sup_{\stackrel[\psi\not=0]{}{\psi \in \Dom(h_V)}}
  \frac{|h_V(\phi_n,\psi) - \lambda \, (\phi_n,\psi)|}{\|\psi\|_{\Dom(h_V)}}
  \leq 
  \frac{\left\|-\Delta\phi_n -\lambda\;\!\phi_n\right\|  
  + \sqrt{\left|v[\phi_n]\right|} \, \sqrt{a}}
  {\sqrt{1-a}}
  \,,
$$
where the right hand side tends to zero due to~\eqref{Weyl1} and~\eqref{Weyl2}.

Summing up, given $\lambda \in [0,+\infty)$,
we have shown that the sequence $\{\phi_n\}_{n\in\Nat}$ satisfies
all the hypotheses of Lemma~\ref{Lem.criterion}.
Consequently, $[0,+\infty) \subset \sigma(H_V)$. 
\end{proof}
\begin{proof}[Proof of Theorem~\ref{Thm.main}]
To conclude, Theorem~\ref{Thm.main} follows as a consequence 
of Theorems~\ref{Thm.evs}, \ref{Thm.cont1}, \ref{Thm.cont2}
and the absence of the residual spectrum
justified already in Section~\ref{Sec.intro}.
\end{proof}

\section{Absence of eigenvalues: the method of multipliers}\label{Sec.mult}
In this last section, we prove Theorems~\ref{Thm.any}, \ref{thm:radi}
and~\ref{thm:radimagn} by a completely different approach 
in comparison with the previous sections.
Namely, we extend the method of multipliers 
developed in the self-adjoint context in~\cite{Barcelo-Vega-Zubeldia_2013}
to complex-valued potentials. 
Here we proceed in all dimensions $d \geq 3$.

Let us consider the equation
\begin{equation}\label{eq:main}
  \Delta u + \lambda u=f,
\end{equation}
where~$\lambda$ is any complex constant;
we write $\lambda_1 := \Re\lambda$ and $\lambda_2 := \Im\lambda$.
Given a measurable function $f : \Real^d \to \Com$
that we assume to merely belong to $H^{-1}(\Real^d)$,
we say that~$u$ is a \emph{solution} of~\eqref{eq:main} 
if $u \in H^1(\Real^d)$ and 
\begin{equation}\label{eq.weak}
  \forall v \in H^1(\Real^d)
  \,, \qquad
  - (\nabla v,\nabla u) + \lambda \, (v,u) = (v,f)
  \,.
\end{equation}
Here, with an abuse of notation, the same symbol $(\cdot,\cdot)$
is used for the inner product in $\sii(\Real^d)$ 
and for the duality pairing between $H^1(\Real^d)$ and $H^{-1}(\Real^d)$
on the left and right hand side of~\eqref{eq.weak}, respectively.
Equation~\eqref{eq:main} is related to the eigenvalue problem of~$H_V$
by setting $f:=Vu$. Notice that any eigenvalue~$\lambda$ of~$H_V$   
necessarily satisfies $\lambda_1>0$ due to~\eqref{Ass}.
If~$u$ is a solution of~\eqref{eq:main}, 
we set 
\begin{equation}\label{eq:uepsilon}
  u^\pm(x) := e^{\pm i \sgn(\lambda_2) \;\! \lambda_1^{\frac12}|x|} \, u(x)
  \,, \qquad
  \sgn(\lambda_2) := 
  \begin{cases}
    \frac{\lambda_2}{|\lambda_2|}
    & \mbox{if} \quad \lambda_2\not=0 \,,
    \\
    1 
    & \mbox{if} \quad \lambda_2=0 \,.
  \end{cases}
\end{equation}

In order to prove Theorem \ref{Thm.any}, 
we establish the following result, 
which shows that~\eqref{eq:main} has no non-trivial solutions
provided that~$f$ is small in a suitable sense. 
\begin{Theorem}\label{Thm.mult}
Let $d \geq 3$.
Let $u$ be a solution of~\eqref{eq:main} 
with $\Re\lambda >0$, and assume that $f$ satisfies
  \begin{equation}\label{eq:assf}
    \|xf\|\leq \Lambda \, \|\nabla u^-\|
  \,,
  \qquad
   \|xf\|\leq \Lambda \, \|\nabla u\|
  \,,
  \end{equation}
where~$\Lambda$ is determined by 
\begin{equation}\label{eq:assf.Lambda}
  \frac{2(2d-3)}{d-2}\,\Lambda+\frac{\sqrt 2}{\sqrt{d-2}}\,\Lambda^{\frac32}<1
  \,.
\end{equation}
Then $u =0$.
\end{Theorem}
\begin{proof}
The proof relies on direct techniques, 
based on multiplication and integration by parts, 
inspired by~\cite{Barcelo-Vega-Zubeldia_2013}, 
in which the methods by \cite{BurqPlanchonStalker04-a, Ikebe-Saito} 
are developed and refined.
Here we propose some slight modifications in the arguments, 
essentially due to the fact that we need to handle complex-valued potentials.
To save space, we abbreviate $\int := \int_{\Real^d}$ 
and omit arguments of integrated functions.

Following~\cite{Barcelo-Vega-Zubeldia_2013},  
we divide the proof into two cases: 
$|\lambda_2|\leq\lambda_1$ and $|\lambda_2|>\lambda_1$.

\paragraph{\fbox{Case $|\lambda_2|\leq\lambda_1$.}} 
Our first step consists in approximating solutions of~\eqref{eq.weak} 
by a standard cutoff and mollification argument, 
which is fundamental to make rigorous the proof in the sequel.
To this aim, let $\xi_R:\mathbb R^d\to[0,1]$ be a smooth function such that
\begin{equation}\label{eq:xi}
  \xi = 1\ \text{in }B_R \,,
  \qquad
  \xi = 0\ \text{in }\mathbb R^d\setminus B_{2R} \,,
  \qquad
  |\nabla\xi_R|\leq 2R^{-1} \,,
  \qquad
  |\Delta\xi_R|\leq 2R^{-1}|x|^{-1} \,,
\end{equation}
for any $R>0$ sufficiently large, 
where $B_R:=\{|x| < R\}$. 
For a function $g:\mathbb R^d\to\mathbb C$, we then denote $g_R:=g\,\xi_R$.
If $u\in H^1(\mathbb R^d)$ is a solution to \eqref{eq.weak}, 
we see that $u_R\in H^1(\mathbb R^d)$ solves
\begin{equation}\label{eq:mainR}
  \Delta u_R + \lambda u_R=
  f_R-2\nabla\xi_R\cdot\nabla u-u\Delta\xi_R=:\widetilde f_R
\end{equation}
in the weak sense of~\eqref{eq.weak}.
Notice that, since $f$ satisfies conditions \eqref{eq:assf} and \eqref{eq:assf.Lambda}, 
we have
\begin{equation}\label{eq:newright}
  \|x\widetilde f_R\|\leq\Lambda
\left\|\nabla u_{R}^-\right\|+\epsilon^2(R) \,,
  \qquad
  \|x\widetilde f_R\|\leq\Lambda\left\|\nabla u_{R}\right\|+\epsilon^2(R) \,,
  \qquad
  \lim_{R\to\infty}\epsilon^2(R)=0 \,.
\end{equation}
Indeed, by \eqref{eq:xi},
\begin{equation*}
   \|x\widetilde f_R\|\leq\|xf_R\|
  +8\left(\int_{R<|x|<2R}|\nabla u|^2\right)^{\frac12}
   +4R^{-2}\left(\int_{R<|x|<2R}|u|^2\right)^{\frac12}
  \,,
\end{equation*}
where the last two terms tends to~$0$ as $R\to\infty$, 
since $u\in H^1(\mathbb R^d)$.

Let now $\phi\in C_0^\infty(\mathbb R^d)$ 
be a function such that $\int\phi=1$, 
and define, for any $\delta>0$, 
$\phi_\delta(x):=\delta^{-d}\phi\left(\frac x\delta\right)$. 
If $u\in H^1(\mathbb R^d)$ is a solution to~\eqref{eq.weak}, 
we see that $u_{R,\delta} := u_R \ast \phi_\delta$ solves
\begin{equation*}
  \Delta u_{R,\delta} + \lambda \, u_{R,\delta} = \widetilde f_{R,\delta}
\end{equation*}
in the weak sense of~\eqref{eq.weak}, where
$\widetilde f_{R,\delta} := \widetilde f_{R}\ast\phi_\delta$. 
More specifically,
\begin{equation}\label{eq:maindeltar}
  \forall v\in H^1(\mathbb R^d) \,,
  \qquad
  \left(-\nabla v,\nabla u_{R,\delta}\right)
  + \lambda\,\left(v,u_{R,\delta}\right)
  = \big(v,\widetilde f_{R,\delta}\big)
  \,.
\end{equation}
By \eqref{eq:newright}, it turns out that
\begin{equation}\label{eq:newright2}
  \|x\widetilde f_{R,\delta}\|\leq\Lambda
\left\|\nabla u_{R,\delta}^-\right\|+\epsilon^2(R) \,,
  \qquad
  \|x\widetilde f_{R,\delta}\|\leq\Lambda\left\|\nabla u_{R,\delta}\right\|
  +\epsilon^2(R) \,,
  \qquad
  \lim_{R\to\infty}\epsilon^2(R)=0 \,,
\end{equation}
where
$
  u_{R,\delta}^-
  := u_{R}^-\ast\phi_\delta
$ 
and $\Lambda$ as in \eqref{eq:assf.Lambda}.

We can now start with suitable algebraic manipulations of equation~\eqref{eq:maindeltar},
which suitably approximates~\eqref{eq.weak}. 
Let $G_1,G_2,G_3:\mathbb R^d\to\mathbb R$ be three smooth functions.
Choosing $v:=G_1 u_{R,\delta}$ in~\eqref{eq:maindeltar}, integrating by parts
and taking the real part of the resulting identity, 
we arrive at the identity
  \begin{equation}\label{eq:id1}
    \lambda_1 \int G_1 |u_{R,\delta}|^2
    -\int G_1 |\nabla u_{R,\delta}|^2
    +\frac12 \int \Delta G_1 \, |u_{R,\delta}|^2 
    =
    \Re\int \widetilde f_{R,\delta} \, G_1 \, \overline{u_{R,\delta}}
    \,.
  \end{equation}
Analogously, choosing~$v:=G_2 u$ in~\eqref{eq.weak} 
and taking the imaginary part of the resulting identity, 
we obtain
\begin{equation}\label{eq:id2}
  \lambda_2\int G_2 |u_{R,\delta}|^2
  -\Im\int \nabla G_2 \cdot \overline{u_{R,\delta}}\nabla u_{R,\delta} 
  =
  \Im \int \widetilde f_{R,\delta} \, G_2 \,\overline{u_{R,\delta}}
  \,,
\end{equation}
where the dot denotes the scalar product in~$\Real^d$.  
Finally, choosing 
$v:=2\nabla G_3 \cdot \nabla u_{R,\delta} + \Delta G_3 \, u_{R,\delta}$ 
in~\eqref{eq.weak}, integrating by parts
and taking the real part of the resulting identity, 
we get
\begin{align}\label{eq:id3}
&
  \int\nabla u_{R,\delta} \cdot \nabla^2 G_3 \cdot \nabla\overline{u_{R,\delta}}
  -\frac14\int \Delta^2 G_3 \, |u_{R,\delta}|^2
  +\lambda_2 \, \Im\int \nabla G_3 \cdot u_{R,\delta}\nabla\overline{u_{R,\delta}}
  \\
  &
  \ \ \ 
  =
  -\frac12\Re\int \widetilde f_{R,\delta} \, \Delta G_3 \, \overline{u_{R,\delta}}
  -\Re\int \widetilde f_{R,\delta} \, \nabla G_3 \cdot \nabla\overline{u_{R,\delta}}
  \,,
  \nonumber
\end{align}
where $\nabla^2 G_3$ denotes the Hessian matrix of~$G_3$
and $\Delta^2 := \Delta\Delta$ is the bi-Laplacian.
Notice that identities \eqref{eq:id1}, \eqref{eq:id2}, \eqref{eq:id3} are justified, 
since $u_{R,\delta}\in C_0^\infty(\mathbb R^d)$ and $G_1,G_2,G_3$ are smooth, 
therefore bounded, together with their derivatives of any order, 
inside the support of $u_{R,\delta}$.

In the following, we assume that $G_1,G_2,G_3$ are radial, 
\ie\ there exist smooth functions $g_1,g_2,g_3: [0,\infty) \to \Real$ 
such that $G_i(x)=g_i(|x|)$ for all $x \in \Real^d$ and $i \in \{1,2,3\}$.
Then 
$$
  \nabla G_i(x) = g_i'(|x|) \frac{x}{|x|}
  \,, \quad
  \Delta G_i(x) = g_i''(|x|) + g_i'(|x|) \frac{d-1}{|x|}
  \,, \quad
  \nabla^2 G_i(x) = g_i''(|x|) \frac{xx}{|x|^2}
  + \frac{g_i'(|x|)}{|x|} \left(I - \frac{xx}{|x|^2} \right)
  \,,
$$
where~$I$ denotes the identity on~$\Real^d$ 
and~$xx$ is the dyadic product of~$x$ and~$x$. 
For any $g:\mathbb R^d\to\mathbb C$, denote by
$$
  \partial_r g(x) := \frac{x}{|x|} \cdot \nabla g(x) 
  \qquad \mbox{and} \qquad
  \nabla_{\!\tau} g(x) := \left(I - \frac{xx}{|x|^2} \right) \cdot \nabla g(x) 
$$
the radial derivative and the angular gradient of~$g$, respectively,
so that 
$
   |\nabla g|^2 = |\partial_r g|^2 + |\nabla_{\!\tau} g|^2
$. 

Taking the sum 
\eqref{eq:id1} + $\lambda_1^{\frac12}$\eqref{eq:id2} + \eqref{eq:id3}, 
we obtain
\begin{align}
\label{eq:again}
  &
  \int|\partial_ru_{R,\delta}|^2(g_3''-g_1) 
  + \int|\nabla_{\!\tau} u_{R,\delta}|^2\left(\frac{g_3'}{|x|}-g_1\right)
  +\int|u_{R,\delta}|^2\left(\lambda_1 g_1 + \lambda_2\lambda_1^{\frac12} g_2\right)
  \nonumber \\
  & \ \ \ 
  +\int|u_{R,\delta}|^2\left(\frac12\Delta G_1 -\frac14\Delta^2 G_3\right)
  -\lambda_1^{\frac12}\Im\int\overline{u_{R,\delta}}\nabla u_{R,\delta}\cdot\nabla G_2 +
  \lambda_2\,\Im\int u_{R,\delta}\,\nabla \overline{u_{R,\delta}}\cdot\nabla G_3
  \nonumber
  \\
  &
  =
  \Re\int \widetilde f_{R,\delta} \, G_1 \,\overline{u_{R,\delta}}
  +\lambda^{\frac12}\Im\int \widetilde f_{R,\delta} \, G_2 \, \overline{u_{R,\delta}}
  -\frac12 \, \Re\int \widetilde f_{R,\delta}\,\overline{u_{R,\delta}} \, \Delta G_3
  -\Re\int \widetilde f_{R,\delta}\,\nabla \overline{u_{R,\delta}}\cdot\nabla G_3
  \,.
\end{align}
Choosing $g_1:=\frac12 g_3''$
and $g_2 := \sgn(\lambda_2) \, g_3'$,
the last identity becomes
\begin{align*}
& 
\frac12\int g_3''
\left(|\partial_ru_{R,\delta}|^2+\lambda_1 \, |u_{R,\delta}|^2\right)
-\sgn(\lambda_2) \, \lambda_1^{\frac12} \,
\Im \int g_3'' \, \overline{u_{R,\delta}} \, \partial_r u_{R,\delta}
+\int|\nabla_{\!\tau} u_{R,\delta}|^2\left(\frac{g_3'}{r}-\frac{g_3''}{2}\right)
\\
& \ \ \ 
+\frac14
  \int|u_{R,\delta}|^2\left(\Delta G_3''-\Delta^2 G_3\right)
  +|\lambda_2| \, \lambda_1^{\frac12}\int g_3'\, |u_{R,\delta}|^2
  +\lambda_2 \, \Im\int g_3' \, u_{R,\delta} \, \partial_r\overline{u_{R,\delta}}
 \\
 & 
 =\frac12\,\Re\int \widetilde f_{R,\delta} \, g_3'' \, \overline{u_{R,\delta}}
  +\lambda_1^{\frac12}\sgn(\lambda_2) \,
  \Im\int \widetilde f_{R,\delta} \, g_3' \, \overline{u_{R,\delta}}
  -\frac12 \, \Re\int \widetilde f_{R,\delta} \, \overline{u_{R,\delta}} \, \Delta G_3
  -\Re\int \widetilde f_{R,\delta} \, \nabla\overline{u_{R,\delta}}\cdot\nabla G_3
  \,,
\end{align*}
where $G_3''(x) := g_3''(|x|)$.
Choosing now $G_3(x) := |x|^2$,
the tangential and radial derivatives of~$u$ sum up and we obtain
\begin{align}\label{last}
& 
\int\left(|\nabla u_{R,\delta}|^2+\lambda_1 \, |u_{R,\delta}|^2\right)
-2\sgn(\lambda_2) \, \lambda_1^{\frac12} \,
\Im\int \overline{u_{R,\delta}} \, \partial_r u_{R,\delta}
  +2 \, |\lambda_2| \, \lambda_1^{\frac12} \int|x||u_{R,\delta}|^2
  +2\lambda_2\,\Im\int|x| \, u_{R,\delta} \, \partial_r\overline{u_{R,\delta}}
 \nonumber \\
 & 
 =(1-d) \, \Re\int \widetilde f_{R,\delta} \, \overline{u_{R,\delta}}
  +2\lambda_1^{\frac12}\sgn(\lambda_2) \,
  \Im\int \widetilde f_{R,\delta}\,|x|\,\overline{u_{R,\delta}}
  -2\,\Re\int \widetilde f_{R,\delta} \, x \cdot \nabla\overline{u_{R,\delta}}
  \,.
\end{align}
Using 
\begin{equation}\label{using}
|\nabla u^-_{R,\delta}|^2=
\left|\nabla u_{R,\delta}
-i\sgn(\lambda_2)\,\lambda_1^{\frac12}\frac{x}{|x|}u_{R,\delta}\right|^2
=
|\nabla u_{R,\delta}|^2+\lambda_1|u_{R,\delta}|^2-2\sgn(\lambda_2)\lambda_1^{\frac12}
\Im \left(\overline{u_{R,\delta}}\,\partial_r u_{R,\delta}\right)
  \,,
\end{equation}
we can rewrite~\eqref{last} as follows:
\begin{align*}
  & \int |\nabla u^-_{R,\delta}|^2
  +2\,|\lambda_2|\,\lambda_1^{\frac12}\int|x||u_{R,\delta}|^2
  +2\lambda_2\,\Im\int|x|\,u_{R,\delta} \, \partial_r\overline{u_{R,\delta}}
 \\
 & 
 =(1-d)\,\Re\int \widetilde f_{R,\delta}\,\overline{u_{R,\delta}}
  +2\lambda_1^{\frac12}\sgn(\lambda_2)\,\Im\int \widetilde f_{R,\delta}\,|x|\,\overline{u_{R,\delta}}
  -2\Re\int \widetilde f_{R,\delta} \, x \cdot \nabla\overline u_{R,\delta}
  \,.
\end{align*}

Subtracting from the last identity
equation~\eqref{eq:id1} with the choice 
$G_1(x):=|\lambda_2|\,\lambda_1^{-\frac12}|x|$,
we arrive at 
\begin{align*}
  & \int|\nabla u^-_{R,\delta}|^2
  -\frac{(d-1)}{2}\,|\lambda_2|\,\lambda_1^{-\frac12}\int\frac{|u_{R,\delta}|^2}{|x|}
  +|\lambda_2| \, \lambda_1^{\frac12}\int|x||u_{R,\delta}|^2
  \\
  & 
  \ \ \ 
  +|\lambda_2|\,\lambda_1^{-\frac12}\int|x||\nabla u_{R,\delta}|^2
  +2\lambda_2\,\Im\int |x| \, u_{R,\delta} \, \partial_r\overline{u_{R,\delta}}
  \\
  & 
=(1-d) \, \Re\int \widetilde f_{R,\delta} \, \overline u_{R,\delta}
  +2\lambda_1^{\frac12}\sgn(\lambda_2)\,
  \Im\int \widetilde f_{R,\delta}\,|x|\,\overline{u_{R,\delta}}
  -2 \, \Re\int \widetilde f_{R,\delta} \, x \cdot \nabla\overline{u_{R,\delta}}
  -|\lambda_2| \, \lambda_1^{-\frac12}
  \, \Re\int \widetilde f_{R,\delta}\,|x|\,\overline{u_{R,\delta}}
  \,.
\end{align*}
Using~\eqref{using} again,
we obtain the key identity
\begin{multline}
\label{eq:id4}
  I := \int \left|\nabla u^-_{R,\delta}\right|^2
  +\frac{|\lambda_2|}{\lambda_1^{\frac12}}\int |x|\left|\nabla u^-_{R,\delta}\right|^2
  -\frac{(d-1)}{2}\frac{|\lambda_2|}{\lambda_1^{\frac12}}\int\frac{|u_{R,\delta}|^2}{|x|}
  \\
  =
  \underbrace{(1-d)\,\Re\int \widetilde f_{R,\delta} \, \overline{u_{R,\delta}}}_{I_1}
  \underbrace{-2\,\Re\int |x|\,\widetilde f_{R,\delta}
  \left(\partial_r\overline{u_{R,\delta}}
  +i\sgn(\lambda_2)\lambda_1^{\frac12}\overline{u_{R,\delta}}\right)}_{I_2}
  \underbrace{-\frac{\lambda_2}{\lambda_1^{\frac{1}{2}}}
  \Re\int|x| \, \widetilde f_{R,\delta} \, \overline{u_{R,\delta}}}_{I_3}
\,.
\end{multline}

By the weighted Hardy inequality
  \begin{equation}\label{eq:hardyweight}
    \forall\psi\in C^\infty_0(\mathbb R^d) \,, \qquad
    \int\frac{|\psi|^2}{|x|}\leq\frac{4}{(d-1)^2}\int|x||\nabla \psi|^2
  \,,
  \end{equation}
and the facts that $u_{R,\delta}\in  C^\infty_0(\mathbb R^d)$ 
and $|u_{R,\delta}|=|u^-_{R,\delta}|$, 
we easily bound the left hand side of~\eqref{eq:id4} 
from below by a positive quantity as follows
\begin{equation}\label{eq:new}
   I
   \geq
   \int \big|\nabla u^-_{R,\delta}\big|^2
   +\frac{|\lambda_2|}{\lambda_1^{\frac12}}\frac{d-3}{d-1}
  \int|x||\nabla u^-_{R,\delta}|^2
  \,.
\end{equation}
We proceed by estimating the individual terms 
on the right hand side of~\eqref{eq:id4}
by means of $\|\nabla u^-_{R,\delta}\|^2$.
By the Schwarz inequality, the Hardy inequality~\eqref{Hardy}
and thanks to~\eqref{eq:newright2}, we have
\begin{equation}\label{eq:est1}
  |I_1| \leq(d-1) \, \|x\widetilde f_{R,\delta}\| \left\|\frac{u_{R,\delta}}{|x|}\right\|
  = (d-1) \, \|x\widetilde f_{R,\delta}\| 
  \left\|\frac{u^-_{R,\delta}}{|x|}\right\|
  \leq\frac{2(d-1)}{d-2}
  \left(\Lambda\,\|\nabla u^-_{R,\delta}\|^2
  +\epsilon^2(R) \, \|\nabla u^-_{R,\delta}\|\right)
  \,.
\end{equation}
Since
$
  \partial_r \overline u_{R,\delta}
  +i\lambda_1^{\frac12}\sgn(\lambda_2)\,\overline{u_{R,\delta}}
  =\overline{\partial_r u^-_{R,\delta}}
$, 
we may write
\begin{equation}\label{eq:est2}
  |I_2| 
  \leq 2\,\|x\widetilde f_{R,\delta}\| \|\partial_r u^-_{R,\delta}\|
  \leq 2\,\|x\widetilde f_{R,\delta}\| \|\nabla u^-_{R,\delta}\|
  \leq 2\left(\Lambda\,\|\nabla u^-_{R,\delta}\|^2
  +\epsilon^2(R) \, \|\nabla u^-_{R,\delta}\|\right)
  \,.
\end{equation}
If $\lambda_2 \not= 0$,
we also need to estimate the term~$I_3$.
First notice that identity~\eqref{eq:id2}
with the constant choice $G_2(x) := \frac{\lambda_2}{|\lambda_2|}$, 
immediately gives the $L^2$-bound
\begin{equation*}
  \|u_{R,\delta}\|^2 
  \leq |\lambda_2|^{-1} \int |\widetilde f_{R,\delta}| |u_{R,\delta}| 
  \,.
\end{equation*}
As a consequence, since $|\lambda_2|\leq\lambda_1$, 
we have
\begin{align}
|I_3| &\leq
\frac{|\lambda_2|}{\lambda_1^{\frac12}} \|x\widetilde f_{R,\delta}\| \|u_{R,\delta}\|
\leq
\left(\Lambda\,\|\nabla u^-_{R,\delta}\|+\epsilon^2(R)\right) 
\sqrt{\int |\widetilde f_{R,\delta}| |u_{R,\delta}|} 
\nonumber
\\
&
\leq
\left(\Lambda\,\|\nabla u^-_{R,\delta}\|+\epsilon^2(R)\right)  
\|x\widetilde f_{R,\delta}\|^{\frac12} \left\|\frac{u_{R,\delta}}{|x|}\right\|^\frac12
\nonumber
\\
& 
\leq
\Lambda^{\frac32}\frac{\sqrt 2}{\sqrt{d-2}} \,
\|\nabla u^-_{R,\delta}\|^2
+
\epsilon^2(R)\frac{\sqrt 2}{\sqrt{d-2}} \, \|\nabla u^-_{R,\delta}\|
\left(\Lambda^{\frac12} \, \|\nabla u^-_{R,\delta}\|
  +\epsilon(R)\right)
  \,.
  \label{eq:est3}
\end{align}

Applying the estimates~\eqref{eq:new}, 
\eqref{eq:est1}, \eqref{eq:est2} and~\eqref{eq:est3} in~\eqref{eq:id4}, 
we obtain
\begin{align*}
&
  \left(
  1-\frac{2(2d-3)}{d-2}\Lambda-\frac{\sqrt 2}{\sqrt{d-2}}\,\Lambda^{\frac32}
  \right)
  \int|\nabla u^-_{R,\delta}|^2
  +\frac{|\lambda_2|}{\lambda_1^{\frac12}}
  \frac{d-3}{d-1}\int|x||\nabla u^-_{R,\delta}|^2
  \\
  &
  \leq
  \epsilon^2(R) \, \big\|\nabla u^-_{R,\delta}\big\|
  \left(\frac{4d-6}{d-2}+
  \frac{\sqrt 2}{\sqrt{d-2}}\,\Lambda^{\frac12}\|\nabla u_{R,\delta}^-\|
  -\frac{\sqrt 2}{\sqrt{d-2}}\,\epsilon(R)\right)
  \,.
\end{align*}
For fixed $R$, let $\delta\to0$ in the last inequality; 
since $u_{R,\delta}$ is compactly supported, 
by the dominated convergence theorem, one gets
\begin{align*}
&
  \left(
  1-\frac{2(2d-3)}{d-2}\Lambda-\frac{\sqrt 2}{\sqrt{d-2}}\,\Lambda^{\frac32}
  \right)
  \int|\nabla u^-_{R}|^2
  +\frac{|\lambda_2|}{\lambda_1^{\frac12}}\frac{d-3}{d-1}\int|x||\nabla u^-_{R}|^2
  \\
  &
  \leq
  \epsilon^2(R) \, \big\|\nabla u^-_{R}\big\|
  \left(\frac{4d-6}{d-2}+
  \frac{\sqrt 2}{\sqrt{d-2}}\,\Lambda^{\frac12}\|\nabla u_{R}^-\|
  -\frac{\sqrt 2}{\sqrt{d-2}}\,\epsilon(R)\right)
  \,.
\end{align*}
Let finally $R\to\infty$; 
by the monotone convergence theorem
and the fact that $u^-\in H^1(\mathbb R^d)$, we conclude that
\begin{equation}\label{eq:finally}
  \left(
  1-\frac{2(2d-3)}{d-2}\Lambda-\frac{\sqrt 2}{\sqrt{d-2}}\,\Lambda^{\frac32}
  \right)
  \int|\nabla u^-|^2
  +\frac{|\lambda_2|}{\lambda_1^{\frac12}}\frac{d-3}{d-1}\int|x||\nabla u^-|^2\leq0
  \,.
\end{equation}
By virtue of~\eqref{eq:assf.Lambda}, 
it follows that~$u^-$ and thus~$u$  
are identically equal to zero. 

\paragraph{\fbox{Case $|\lambda_2|>\lambda_1$.}} 
Let $u\in H^1(\mathbb R^d)$ be a solution to \eqref{eq.weak}.
Choosing as a test function $v=\pm u$ in \eqref{eq.weak}, and taking real and imaginary parts of the resulting identities, one easily gets
\begin{equation}\label{eq:outside}
  (\lambda_1\pm\lambda_2)\int|u|^2
  =\int|\nabla u|^2+\Re\int f\overline u\pm\Im\int f\overline u.
\end{equation}
By the Schwarz inequality,
the Hardy inequality~\eqref{Hardy}
and assumption~\eqref{eq:assf},
we estimate
\begin{equation*}
  \Re\int f\overline u\pm\Im\int f\overline u
  \leq2\int|f||u|\leq2\,\|xf\|\left\|\frac{u}{|x|}\right\|\leq
  \frac{4}{d-2}\,\Lambda\int|\nabla u|^2
  \,.
\end{equation*}
Consequently, \eqref{eq:outside}~yields
\begin{equation*}
  (\lambda_1\pm\lambda_2)\int|u|^2
  \geq\left(1-\frac{4}{d-2}\Lambda\right)\int|\nabla u|^2
  \,.
\end{equation*}
Notice that \eqref{eq:assf.Lambda} implies that $\Lambda<\frac{d-2}{4}$, 
therefore the last inequality forces $\lambda_1\pm\lambda_2\geq0$ 
unless $u=0$. 
Since we assume $|\lambda_2|>\lambda_1$, we conclude that $u=0$.
\end{proof}

By taking $f:=Vu$ in Theorem~\ref{Thm.mult}
(notice that $Vu$ belongs to $H^{-1}(\Real^d)$
under the hypothesis~\eqref{Ass.bis.better})
and using that $|u|=|u^-|$, 
we immediately obtain
\begin{Corollary}\label{Corol.mult}
Let $d \geq 3$ and suppose
\begin{equation}\label{Ass.bis.better}
  \forall \psi \in H^1(\Real^d) 
  \,, \qquad
  \int_{\Real^d} |x|^2 \, |V(x)|^2 \, |\psi(x)|^2 \, dx
  \leq \Lambda^2 \int_{\Real^d} |\nabla\psi|^2
  \,,
\end{equation}
where~$\Lambda$ satisfies~\eqref{eq:assf.Lambda}.
Then 
$ 
  \sigma_\mathrm{p}(H_V) = \varnothing
$.
\end{Corollary}
\begin{proof}
In fact, Theorem~\ref{Thm.mult} only gives the weaker conclusion
that no complex point~$\lambda$ satisfying $\Re\lambda> 0$ 
can be an eigenvalue of~$H_V$.
However, \eqref{Ass.bis.better} with~\eqref{eq:assf.Lambda} implies~\eqref{Ass}, 
which in turn yields that all possible eigenvalues of~$H_V$ 
are included in the right complex plane, \ie~$\Re\lambda>0$.
Indeed, this fact follows from the identity
\begin{equation}\label{id1}
  \int |\nabla u|^2 + \Re \int V |u|^2 = \Re\lambda \int |u|^2
  \,, 
\end{equation}
which can be obtained from~\eqref{eq:id1} 
with the constant choice $G_1:=1$ and $f:=Vu$.
\end{proof}

Now we are in a position to prove Theorem~\ref{Thm.any}.
\begin{proof}[Proof of Theorem~\ref{Thm.any}]
Theorem~\ref{Thm.any} follows as a weaker version 
of Corollary~\ref{Corol.mult}.
Indeed, it is easy to see that
any~$\Lambda$ verifying~\eqref{eq:assf.Lambda} 
necessarily satisfies $\Lambda \leq (d-2)/2$.
Using the latter in the former, we obtain~\eqref{Ass.bis} 
as a sufficient condition which guarantees~\eqref{Ass.bis.better}.
\end{proof}

We now turn our attention to Theorem \ref{thm:radi}. 
In analogy with the above strategy, 
we first study the (more difficult) part $\Re\lambda>0$.
In the following, we set $V_1:=\Re V$ and $V_2 := \Im V$.
\begin{Theorem}\label{Thm.mult2}
Let $d \geq 3$.
Let $u \in \nspace(\Real^d)$ be a solution of~\eqref{eq:main} 
with $\Re\lambda >0$, and let $f:=Vu$ where~$V$ satisfies
\eqref{eq:assV12}, \eqref{eq:asspart}, \eqref{eq:assV2} and \eqref{eq:assf.Lambdanew}.
Then $u =0$.
\end{Theorem}
\begin{proof}
The proof is completely analogous to that of Theorem~\ref{Thm.mult}. 
The only difference consists in the way 
we handle the right-hand side of \eqref{eq:id4}, as we see in the sequel.

\paragraph{\fbox{Case $|\lambda_2|\leq\lambda_1$.}} 
With the same notations as above, 
if $u\in \nspace(\Real^d) \subset H^1(\mathbb R^d)$ solves \eqref{eq:main}, 
then identity \eqref{eq:id4} holds. 
We now need to rewrite the right-hand side of \eqref{eq:id4} in a suitable way. 
To this aim, recall that $\widetilde f_R$ is defined via \eqref{eq:mainR}, 
where $f:=Vu$. It is convenient to introduce the notation
\begin{equation}\label{eq:newright22}
  K_R(u,\nabla u):=-2\nabla\xi_R\cdot\nabla u-u\Delta\xi_R.
\end{equation}
so that $\widetilde f_R=f_R+K_R(u,\nabla u)$.
Putting~\eqref{eq:newright22} into~\eqref{eq:id4}, 
integrating by parts in the first two terms involving~$V_1$
and taking the limit as $\delta\to0$, one gets the following key identity:
\begin{multline}
\label{eq:id6}
I:=
  \int \big|\nabla u^-_{R}\big|^2
   +\frac{|\lambda_2|}{\lambda_1^{\frac12}}
  \int|x||\nabla u^-_{R}|^2
  -\frac{(d-1)}{2}\frac{|\lambda_2|}{\lambda_1^{\frac12}}\int\frac{|u_{R}|^2}{|x|}
  +
 \frac{\lambda_2}{\lambda_1^{\frac12}}\int|x|V_1|u_R|^2
  \\
  =
 \underbrace{\int|u_R^-|^2\left(V_1+|x|\partial_rV_1\right)}_{I_1}
 +
 \underbrace{2\Im\int|x|V_2u_R
  \left(\partial_r\overline{u_R}
  +i\sgn(\lambda_2)\lambda_1^{\frac12}\overline{u_R}\right)}_{I_2}
\\
+\underbrace{(1-d)\Re\int K_R(u,\nabla u)\overline{u_R}-2\Re\int|x|K_R(u,\nabla u)
\left(\partial_r\overline{u_R}+i\sgn(\lambda_2)\lambda_1^{\frac12}\overline{u_R}\right)
-\frac{\lambda_2}{\lambda_1^{\frac12}}\Re\int|x|K_R(u,\nabla u)\overline{u_R}}_{I_3}
\,.
\end{multline}

We start by estimating the individual terms 
on the right hand side of~\eqref{eq:id6}. 
Thanks to assumption \eqref{eq:asspart}, we have
\begin{equation}\label{eq:i1}
  I_1
  =
  \int|u_R^-|^2\partial_r(|x|V_1)
  \leq
  \int|u_R^-|^2\left[\partial_r(|x|V_1)\right]_+
  \leq b_2^2\int|\nabla u_R^-|^2.
\end{equation}
We now use
$
  \partial_r \overline u_{R}
  +i\lambda_1^{\frac12}\sgn(\lambda_2)\,\overline{u_{R}}
  =\overline{\partial_r u^-_{R}}
$
to write
\begin{equation}\label{eq:i3}
|I_2|
\leq
2\|xV_2u_R\|\|\partial_ru_R^-\|
\leq
2\|xV_2u_R\|\|\nabla u_R^-\|
\leq2b_3\int|\nabla u_R^-|^2.
\end{equation}
Finally, 
by \eqref{eq:xi} and the fact that $u_R\in H^1(\mathbb R^d)$, 
one easily gets that
\begin{equation}\label{eq:i4}
  |I_3|\leq\epsilon^2(R),
  \qquad
  \lim_{R\to\infty}\epsilon^2(R)=0.
\end{equation}

We now proceed by estimating the left-hand side of \eqref{eq:id6} from below. 
By \eqref{eq:assV12} we obtain
\begin{equation}\label{eq:estv14}
  \frac{|\lambda_2|}{\lambda_1^{\frac12}}\int |x|V_1|u_R|^2
  \geq
  -\frac{|\lambda_2|}{\lambda_1^{\frac12}}\int (V_1)_
  -\left||x|^{\frac12}u_R^-\right|^2
  \geq -b_1^2\frac{|\lambda_2|}{\lambda_1^{\frac12}}
  \int\left|\nabla\left(|x|^{\frac12}u_R^-\right)\right|^2
  \,.
\end{equation}
Now write
\begin{equation}\label{eq:inuova}
  \frac{|\lambda_2|}{\lambda_1^{\frac12}}
  \int|x||\nabla u^-_{R}|^2
  -\frac{(d-1)}{2}\frac{|\lambda_2|}{\lambda_1^{\frac12}}\int\frac{|u_{R}|^2}{|x|}
  =
  \frac{|\lambda_2|}{\lambda_1^{\frac12}}
  \int\left|\nabla\left(|x|^{\frac12}u_R^-\right)\right|^2
  -\frac14\frac{|\lambda_2|}{\lambda_1^{\frac12}}
  \int\frac{\left|u_R^-\right|^2}{|x|}
  \,.
\end{equation}
Notice that identity~\eqref{eq:id2}
with the constant choice $G_2(x) := \frac{\lambda_2}{|\lambda_2|}$, 
in the limit as $\delta\to0$, reads as follows
\begin{equation*}
  |\lambda_2|\int|u_R|^2=\frac{\lambda_2}{|\lambda_2|}\int V_2|u_R|^2+
  \frac{\lambda_2}{|\lambda_2|}\Im\int K_R(u,\nabla u)\overline{u_R}.
\end{equation*}
Since $u_R\in H^1(\mathbb R^d)$, 
arguing as in \eqref{eq:1000}, by \eqref{eq:assV2}, \eqref{eq:xi} 
and the fact that $|u_R|=|u_R^-|$, we obtain the $L^2$-bound
\begin{equation}\label{eq:arguing}
  \|u_R\|^2\leq|\lambda_2|^{-1}\left(\frac{2b_3}{d-2}\int|\nabla u_R^-|^2+\epsilon^2(R)\right),
  \qquad
  \lim_{R\to\infty}\epsilon^2(R)=0.
\end{equation}
As a consequence of \eqref{eq:arguing}, since $|\lambda_2|\leq\lambda_1$, 
we can estimate the last term in \eqref{eq:inuova}, 
by the Schwarz and Hardy inequalities as follows:
\begin{equation}\label{eq:estl2}
  \frac{|\lambda_2|}{\lambda_1^{\frac12}}
  \int\frac{\left|u_R^-\right|^2}{|x|}
  \leq
  \frac{|\lambda_2|}{\lambda_1^{\frac12}}
  \left\|\frac{u_R^-}{|x|}\right\|\left\|u_R^-\right\|
  \leq \sqrt{b_3} \,
  \left(\frac{2}{d-2}\right)^{\frac32}\int\left|\nabla u_R^-\right|^2
  + \frac{2}{d-2} \, \|\nabla u_R^-\| \, |\epsilon(R)| \,,
\end{equation}
where $\epsilon(R)$ is the error term from~\eqref{eq:arguing}.
By \eqref{eq:estv14}, \eqref{eq:inuova}, and \eqref{eq:estl2}, we conclude that
\begin{equation}\label{eq:stimai}
 I
 \geq
 \left[1-\frac{1}{4} \sqrt{b_3} \, \left(\frac{2}{d-2}\right)^{\frac32}\right]
 \int\left|\nabla u_R^-\right|^2
 - \frac{1}{4} \frac{2}{d-2} \, \|\nabla u_R^-\| \, |\epsilon(R)|
  \,.
\end{equation}

Applying \eqref{eq:i1}, \eqref{eq:i3}, \eqref{eq:i4} and \eqref{eq:stimai} 
in \eqref{eq:id6}, we obtain
\begin{equation*}
\left[1-b_2^2-2 \, b_3 
  - \frac{1}{4} \sqrt{b_3} \, \left(\frac{2}{d-2}\right)^{\frac32}\right]
\int\left|\nabla u_R^-\right|^2
 \leq
 \epsilon^2(R) + \frac{1}{4} \frac{2}{d-2} \, \|\nabla u_R^-\| \, |\epsilon(R)|
 \,,
\end{equation*}
for any $R>0$,
with $ \lim_{R\to\infty}\epsilon^2(R)=0$.
In the limit as $R\to\infty$, 
by the monotone convergence theorem, we finally get
\begin{equation}\label{eq:id7}
\left[1-b_2^2-2 \, b_3 
  - \frac{1}{4} \sqrt{b_3} \, \left(\frac{2}{d-2}\right)^{\frac32}\right]
\int\left|\nabla u^-\right|^2\leq0 \,.
\end{equation}
By virtue of~\eqref{eq:assf.Lambdanew}, 
it follows that~$u^-$ and thus~$u$  
are identically equal to zero. 

\paragraph{\fbox{Case $|\lambda_2|>\lambda_1$.}} 
The proofs in this case is based on identity~\eqref{eq:outside}. 
When $f:=Vu$, it reads as follows:
\begin{equation}\label{eq:outsideV}
  (\lambda_1\pm\lambda_2)\int|u|^2
  =
  \int|\nabla u|^2+\int V_1|u|^2\pm\int V_2|u|^2
  \geq
  \int|\nabla u|^2-\int (V_1)_-|u|^2-\left|\int V_2|u|^2\right|.
\end{equation}
By means of \eqref{eq:1000}, \eqref{eq:assV12} and \eqref{eq:assV2}, we have
\begin{equation*}
(\lambda_1\pm\lambda_2)\int|u|^2
\geq
\left[1-b_1^2-\frac{2b_3}{d-2}\right]
\int|\nabla u|^2.
\end{equation*}
Therefore, condition \eqref{eq:assf.Lambdanew} 
implies that $\lambda_1\pm\lambda_2\geq0$, 
and since $|\lambda_2|>\lambda_1$ we conclude that $u$ is identically zero.
\end{proof}

Now we are in a position to prove Theorem~\ref{thm:radi}.
\begin{proof}[Proof of Theorem~\ref{thm:radi}]
Theorem \ref{Thm.mult2} implies that
$\sigma_\mathrm{p}(H_V)\cap\{\lambda_1>0\}=\varnothing$. 
In addition, if $\lambda_1\leq0$, 
then choosing $v:=u$ in~\eqref{eq.weak}
and taking the resulting real part, one obtains
\begin{equation*}
  \lambda_1\int|u|^2
  = \int|\nabla u|^2+\int V_1\,|u|^2
  \geq \int|\nabla u|^2 - \int (V_1)_-\,|u|^2
  \geq (1-b_1^2)\int|\nabla u|^2
  \,,
\end{equation*}
where the last inequality follows by~\eqref{eq:assV12}. 
This implies that $\sigma_\mathrm{p}(H_V)\cap\{\lambda_1\leq 0\}=\varnothing$, 
so the proof is completed. 
\end{proof}
We conclude the manuscript with the proof of Theorem~\ref{thm:radimagn}. 
Since the strategy is identical to the proof of Theorem~\ref{thm:radi}, 
we just sketch it.
\begin{proof}[Proof of Theorem~\ref{thm:radimagn}]
Equation \eqref{eq:main} is now replaced by
\begin{equation}\label{eq:mainmagn}
  \Delta_A u+\lambda u = V u
  \,,
\end{equation}
where $\Delta_A := \nabla_{\!A} \cdot \nabla_{\!A}$.
Let $u\in \nspace_{\!A}(\mathbb R^d)$ be a weak solution to~\eqref{eq:mainmagn}. 
By similar algebraic manipulations as in the proof of Theorem~\ref{thm:radi}, 
we get an analogue to~\eqref{eq:id6}:
\begin{align*}
&
   \int \big|\nabla_{\!A} u^-_{R}\big|^2
   +\frac{|\lambda_2|}{\lambda_1^{\frac12}}\frac{d-3}{d-1}
  \int|x||\nabla_{\!A} u^-_{R}|^2
  \leq
\int|u_R|^2\left(V_1+|x|\partial_rV_1\right)
 +
 \frac{\lambda_2}{\lambda_1^{\frac12}}\int|x|V_1|u_R|^2
 \\
&
\ \ \ 
 +
 2\Im\int|x|u_RV_2\left(\overline{\partial_r^Au_R}
+i\sgn(\lambda_2)\lambda_1^{\frac12}\overline{u_R}\right)
 -2\Im\int|x|u_RB_\tau\cdot\overline{\nabla_{\! A}u_R}
 \\
&
\ \ \ 
+ (1-d)\Re\int K_R(u,\nabla_{\!A} u)\overline{u_R}-2\Re\int|x|K_R(u,\nabla_{\!A} u)
\left(\partial_r^A\overline{u_R}
+i\sgn(\lambda_2)\lambda_1^{\frac12}\overline{u_R}\right)
\\
&
\ \ \ 
-\frac{\lambda_2}{\lambda_1^{\frac12}}\Re\int|x|K_R(u,\nabla_{\!A} u)\overline{u_R}
\,,
\end{align*}
where $\partial_r^A := \frac{x}{|x|}\cdot\nabla_{\!A}$.
In fact, in order to obtain the last identity, 
one proceeds exactly as above, 
with the only difference arising once obtaining identity \eqref{eq:id3}, 
in which we use the test function 
$v:=\nabla G_3\cdot\nabla_Au_{R,\delta}+\Delta G_3u_{R,\delta}$.
The key remark is that~$B_\tau$ is a tangential vector, so that 
$$
B_\tau\cdot\nabla_{\!A}u 
= B_\tau\cdot\left(\nabla_{\!A}u+i\sgn(\lambda_2)\lambda_1^{\frac12}\frac{x}{|x|}u\right)
=B_\tau\cdot\nabla_{\!A}u_R^- \,,
$$
and we can rewrite the last inequality as
\begin{align*}
&
  \int \big|\nabla_A u^-_{R}\big|^2
   +\frac{|\lambda_2|}{\lambda_1^{\frac12}}\frac{d-3}{d-1}
  \int|x||\nabla_A u^-_{R}|^2
  \\
  &
  \leq
\int|u_R|^2\left(V_1+|x|\partial_rV_1\right)
 +
\frac{\lambda_2}{\lambda_1^{\frac12}}\int|x|V_1|u_R|^2
 +
2\Im\int|x|u_R\overline{\nabla_Au_R}\cdot\left(V_2\frac{x}{|x|}-B_\tau\right)
 \\
&
\ \ \ 
+ (1-d)\Re\int K_R(u,\nabla_A u)\overline{u_R}-2\Re\int|x|K_R(u,\nabla_A u)
\left(\partial_r^A\overline{u_R}+i\sgn(\lambda_2)\lambda_1^{\frac12}\overline{u_R}\right)
\\
&
\ \ \ -\frac{\lambda_2}{\lambda_1^{\frac12}}\Re\int|x|K_R(u,\nabla_A u)\overline{u_R}
\,.
\end{align*}
One now proceeds in the same way as in the magnetic-free case, 
to conclude that $u=0$ if $\Re\lambda>0$. 
To complete the proof, we then argue exactly as above; we omit further details.
\end{proof}
%

\subsection*{Acknowledgment}
%
The first two authors acknowledge the hospitality of
the \emph{Basque Center for Applied Mathematics} in Bilbao
where this work was initiated and completed.
The first author also acknowledges the hospitality of 
the \emph{Nuclear Physics Institute} in \v{R}e\v{z}
where this work was developed in part. 
The research was partially supported
by the project RVO61389005, the GACR grant No.\ 14-06818S, 
and the MIUR grant FIRB 2012 -- RBFR12MXPO-002.

%
%

\providecommand{\bysame}{\leavevmode\hbox to3em{\hrulefill}\thinspace}
\providecommand{\MR}{\relax\ifhmode\unskip\space\fi MR }
\providecommand{\MRhref}[2]{%
  \href{http://www.ams.org/mathscinet-getitem?mr=#1}{#2}
}
\providecommand{\href}[2]{#2}

\end{document}